\documentclass[test,preprint]{imsart}

\RequirePackage{amsthm,amsmath,amsfonts,amssymb}
\RequirePackage[numbers]{natbib}

\usepackage[utf8]{inputenc}

\usepackage{hyperref}

\usepackage{latexsym,amsfonts,amsmath,amsthm,amssymb, epsfig}
\usepackage{mathtools}

\usepackage{enumitem}

\usepackage{subcaption}

\usepackage{booktabs}
\usepackage{multirow}

\usepackage{tabularx}
\usepackage{tcolorbox}

\usepackage{makecell}

\usepackage{algorithm}
\usepackage{algorithmic}

\usepackage{hyphenat}

\startlocaldefs

\providecommand*{\textN}[1]{\|{#1}\|} 

\providecommand*{\Var}{{\operatorname{Var}}}

\providecommand*{\Covv}[1]{\operatorname{Cov}\left({#1}\right)}   
\providecommand*{\Uni}[1]{\operatorname{Uni}({#1})}
\providecommand*{\Span}{\operatorname{span}}     

\providecommand{\Dim}{\operatorname{dim}}            
\providecommand{\dim}{\Dim}

\providecommand*{\Span}[1]{\operatorname{Span}\left\{{#1}\right\}}     

\providecommand*{\Range}[1]{\operatorname{Range}({#1})}                
\providecommand{\rank}{\operatorname{rank}}                        

\renewcommand{\Im}{\operatorname{Im}}             
\providecommand{\argmin}{\operatorname*{argmin}}  
\providecommand{\Id}{\Op{Id}}                     

\providecommand{\CC}{{\cal C}}

\providecommand{\CF}{{\cal F}}

\providecommand{\CN}{{\cal N}}
\providecommand{\CO}{{\cal O}}
\providecommand{\CP}{{\cal P}}

\providecommand{\CR}{{\cal R}}
\providecommand{\CS}{{\cal S}}

\providecommand{\CX}{{\cal X}}
\providecommand{\CY}{{\cal Y}}

\providecommand{\bbE}{\mathbb{E}}

\providecommand{\bbN}{\mathbb{N}}

\providecommand{\bbP}{\mathbb{P}}

\providecommand{\bbR}{\mathbb{R}}
\providecommand{\bbS}{\mathbb{S}}

\providecommand{\bbZ}{\mathbb{Z}}

\providecommand*{\N}[1]{\left\|{#1}\right\|} 
\newcommand*{\SN}[1]{\left|{#1}\right|}      

\newcommand*{\Op}[1]{\mathsf{#1}} 

\newcolumntype{H}{>{\setbox0=\hbox\bgroup}c<{\egroup}@{}}

\newcommand\independent{\protect\mathpalette{\protect\independenT}{\perp}}
\def\independenT#1#2{\mathrel{\rlap{$#1#2$}\mkern2mu{#1#2}}}

\newcommand{\centerRV}[1]{\tilde{#1}}
\newcommand{\Trace}{\textrm{Trace}}

\newcommand{\lsmean}[2]{\hat{\bbE}_{#1}{#2}}
\newcommand{\MIMPTrue}{P}
\newcommand{\MIMP}{\hat P}
\newcommand{\MIMPJ}[1]{\hat P_{{#1}}}
\newcommand{\PVX}{PX}
\newcommand{\QVX}{QX}

\newcommand{\minRho}{\rho_{J,\min}}

\newcommand{\maxStandardNorm}{\omega_{J,\max}}

\newcommand{\ra}[1]{\renewcommand{\arraystretch}{#1}}
\newcommand{\MSE}[2]{\textrm{MSE}\left({#1}, {#2}\right)}

\newlength\stextwidth
\newcommand\boxwidth[3][c]{%
  \settowidth{\stextwidth}{#2}%
  \makebox[\stextwidth][#1]{#3}%
}

\allowdisplaybreaks

\theoremstyle{plain}
\newtheorem{theorem}{Theorem}
\newtheorem{proposition}[theorem]{Proposition}
\newtheorem{lemma}[theorem]{Lemma}
\newtheorem{corollary}[theorem]{Corollary}

\theoremstyle{definition}
\newtheorem{definition}[theorem]{Definition}
\newtheorem{example}{Example}
\newtheorem{remark}[theorem]{Remark}

\endlocaldefs

\begin{document}

\begin{frontmatter}
\title{Estimating multi-index models \\ with response-conditional least squares}
\runtitle{Response-conditional least squares}

\begin{aug}
\author[A]{\fnms{Timo} \snm{Klock}\ead[label=e1]{timo@simula.no}},
\author[B]{\fnms{Alessandro} \snm{Lanteri}\ead[label=e2]{alessandro.lanteri@unito.it}}
\and
\author[C]{\fnms{Stefano} \snm{Vigogna}\ead[label=e3]{vigogna@dibris.unige.it}}
\address[A]{Machine Intelligence Department, Simula Research Laboratory, \printead{e1}}

\address[B]{ESOMAS, Universit\`a degli Studi di Torino and Collegio Carlo Alberto, \printead{e2}}

\address[C]{MaLGa Center, DIBRIS, Universit\`a degli Studi di Genova, \printead{e3}}
\end{aug}

\begin{abstract}
The multi-index model is a simple yet powerful high-dimensional regression model
which circumvents the curse of dimensionality
assuming $ \bbE [ Y | X ] = g(A^\top X) $
for some unknown index space $A$ and link function $g$.
In this paper we introduce a method for the estimation of the index space,
and study the propagation error of an index space estimate in the regression of the link function.
The proposed method approximates the index space
by the span of linear regression slope coefficients computed over level sets of the data.
Being based on ordinary least squares,
our approach is easy to implement and computationally efficient.
We prove a tight concentration bound that shows $N^{-1/2}$-convergence,
but also faithfully describes the dependence on the chosen partition of level sets,
hence giving indications on the hyperparameter tuning.
The estimator's competitiveness is confirmed
by extensive comparisons with state-of-the-art methods, both on synthetic and real data sets.
As a second contribution,
we establish minimax optimal generalization bounds for k-nearest neighbors and piecewise polynomial regression
when trained on samples projected onto any $N^{-1/2}$-consistent estimate of the index space,
thus providing complete and provable estimation of the multi-index model.
\end{abstract}



\end{frontmatter}


\section{Introduction}
\label{sec:sdr_intro}
Many recent advances in the analysis of high\hyp{}dimensional data
are based on the observation that
real-world data are inherently structured,
and the relationship between variables, features and responses is often of a lower dimensional nature \cite{adragni2009sufficient,Bickel2007,NIPS2011_4455,NIPS2013_5103,liao2016learning,GMRARegression}.
A popular model incorporating this structural assumption is the \emph{multi-index model},
which poses the relation between a predictor $ X \in \bbR^{D} $ and a response $ Y \in \bbR $ as
\begin{align}
\label{eq:MIM}
Y = g(A^\top X) + \zeta,
\end{align}
where $A \in \bbR^{D\times d}$ is an unknown full column rank matrix  with $d\ll D$,
$ g : \bbR^d \to \bbR $ is an unknown function,
and $\zeta$ is a noise term with $\bbE[\zeta|X] = 0$, independent of $X$ given $A^\top X$.
In the following we refer to $g$ as the \emph{link function}
and $A$ as the \emph{index space},
assuming, without loss of generality, that the columns of $A$ are orthonormal \cite{fornasier2012learning}.
Model \eqref{eq:MIM} asserts that the information required to predict the conditional expectation $f(x) := \bbE[Y|X=x]$
is encoded in the distribution of $A^\top X$.
Therefore, knowing the projection $ P := AA^\top $ allows to estimate $f$ in a nonparametric fashion
with a number of samples scaling with the intrinsic dimension $d$, rather than the ambient dimension $D$.

Ways to estimate the index space have been studied extensively over
the years and by now several methods have been proposed.
Most of them originate from
the statistical literature, starting with the seminal work of \cite{li1991sliced}
and going forward with \cite{dennis2000save,
li2007directional,li2005contour}.
In recent years,
the problem has gained popularity also in the machine learning community, due to its relation and similarity
to (shallow) neural networks \cite{fornasier2012learning, fornasier2018identification,hemant2012active}.
Despite the variety of available approaches, there is no distinctly \emph{best} method:
some estimators are better suited for practical purposes as they are computationally efficient and easy to implement, while others generally enjoy better theoretical guarantees.
We provide an extensive overview in Section \ref{subsec:overview_index_space_estimation}.

In this work we derive and analyze a method for estimating $f$ under the model assumption \eqref{eq:MIM} from a given data set $\{(X_i,Y_i) : i = 1,\dots,N\}$, where $(X_i,Y_i)$ are independent copies of $(X,Y)$.
First, we construct an estimate $\MIMP$ of the projection $\MIMPTrue$ based on the span of response-conditional least-squares vectors of the data.
Once $\MIMP$ has been computed, the second step is a regression task on the reduced data set $\{(\MIMP X_i, Y_i) : i = 1,\dots,N\}$, which can
be solved by classical nonparametric estimators such as piecewise polynomials or kNN-regression.
The proposed method is attractive for practitioners due to its simplicity and
efficiency, with almost no parameter adjustment needed.
Furthermore, it is provable,
with strong theoretical guarantees neatly derivable from a few reasonable assumptions.
We establish tight concentration bounds describing
the estimator's performance in the finite sample regime.
In particular, we prove that $\Vert \MIMP - \MIMPTrue\Vert \in \CO(N^{-1/2})$,
and determine the explicit dependence of the constants on the parameters involved.
A data-driven approximation of the index space error empirically confirms the tightness
of our concentration bound,
providing guidance for hyperparameter tuning.
Moreover, to the best of our knowledge,
we
are the first ones to
provide generalization
guarantees for model \eqref{eq:MIM} that take into account the propagation of the projection error $\Vert \MIMP - \MIMPTrue\Vert$ into the reduced regression problem.
Specifically, we analyze two popular regression
methods,
namely k-nearest neighbors regression (kNN) and piecewise polynomial regression, and show that
the minimax optimal $d$-dimensional estimation rate is achieved if $\MIMP$ is \emph{any} index estimate such that $\Vert \MIMP - \MIMPTrue\Vert \in \CO(N^{-1/2})$.

\subsection{Related work on index space estimation}
\label{subsec:overview_index_space_estimation}
Many methods for estimating the index space have been developed in
the statistical literature under the name of \emph{sufficient dimension reduction} \cite{Li2018},
where the multi-index model is relaxed to
\begin{align}
\label{eq:EDR_assumption}
&Y \independent X | A^\top X .
\end{align}
Note that this setting generalizes our problem since \eqref{eq:MIM} and $\varepsilon \independent X | A^\top X$ imply \eqref{eq:EDR_assumption}.
A space $\Im(A)$ satisfying \eqref{eq:EDR_assumption} is called
a \emph{dimension reduction subspace}, and if
the intersection of such spaces satisfies \eqref{eq:EDR_assumption} it is called
\emph{central subspace}.
Except for degenerate cases, a unique central subspace exists
\cite{cook1994interpretation, cook2009regression}. One can also consider a model where  \eqref{eq:EDR_assumption} is replaced by
$Y \independent \bbE[Y|X] | A^\top X$, which leads to the definition of \emph{central mean subspace} \cite{cook2002dimension}.
In the case of model \eqref{eq:MIM} with $\varepsilon \independent X | A^\top X$, the space $\Im(A)$ is both the central subspace and the central mean
subspace \cite{cook2002dimension}. Thus, we will treat related research under the same umbrella.

The methods for sufficient dimension reduction can broadly be grouped into \emph{inverse regression based methods}
and \emph{nonparametric methods} \cite{adragni2009sufficient,ma2013review}.
The first group reverses the regression dependency between $X$ and $Y$ and uses inverse
statistical moments to construct a matrix $\Lambda$ with $\Im(\Lambda) \subseteq \Im(A)$.
The most prominent representatives are sliced inverse regression (SIR/SIRII) \cite{li1991sliced,SIR-II},
sliced average variance estimation (SAVE) \cite{dennis2000save}, and contour regression/directional regression (CR/DR) \cite{li2007directional,li2005contour}
(see Table \ref{tab:inverse_regression} for the corresponding definition
of $\Lambda$).
Linear combinations of related matrices $\Lambda$ have been
called \emph{hybrid methods} \cite{zhu2007hybrid}. Furthermore, in the case where $X$ follows a normal distribution, two popular
methods are principal Hessian directions (pHd) \cite{li1992principal} and iterative
Hessian transformations (iHt) \cite{cook2002dimension}.
In this setting, $\Lambda$
is the averaged Hessian matrix of the regression function, which can be efficiently
computed using Stein's Lemma.

If $\Im(\Lambda) \subseteq \Im(A)$, eigenvectors corresponding to nonzero eigenvalues of $\Lambda$ yield an unbiased
subspace of the index space $\Im(A)$. A typical assumption to guarantee this
is the \emph{linear conditional mean} (LCM), given by $\bbE[X|PX] = PX$. It holds, for example,
for all elliptically symmetric distributions \cite{li1991sliced,ma2013review}.
Methods based on second order moments usually need in addition
the \emph{constant conditional variance} assumption (CCV),
which requires $\Covv{X|PX}$ to be nonrandom.
In particular, the normal distribution satisfies both LCM and CCV.
If $\Im(\Lambda) = \Im(A),$ a method is called \emph{exhaustive}.
A condition to ensure exhaustiveness is
$\bbE[v^\top Z|Y]$ being non-degenerate (\emph{i.e.} not
almost surely equal to a constant) for all nonzero $v \in \Im(A)$,
where $Z$ is the standardization of $X$.
In Table \ref{tab:inverse_regression} we denote this condition by RCP (random conditional projection),
and by RCP$^2$ when $\bbE[v^\top Z|Y]$ is replaced by $\bbE[(v^\top Z)^2|Y]$.

\begin{table}[h]
\begin{center}
\ra{1.2}
\scriptsize
\resizebox{\columnwidth}{!}{%
\begin{tabular}{l l l l}
    Method & Matrix $\Lambda$ & $\Im(\Lambda) \subseteq \Im(A)$ & $\Im(\Lambda) = \Im(A)$ \\ \midrule
    \boxwidth[l]{SAVE \ \cite{dennis2000save}}{SIR \hfill \cite{li1991sliced}} & $\Covv{\bbE[Z|Y]}$ & LCM &
     RCP
     \\ \midrule
    \boxwidth[l]{SAVE \ \cite{dennis2000save}}{SIRII \hfill \cite{SIR-II}}
    &
    $\bbE(\Covv{Z|Y}-\bbE\Covv{Z|Y})^2$
    &
    LCM and CCV &
    N/A
    \\ \midrule
    SAVE \ \cite{dennis2000save} & $\bbE(\Id - \Covv{Z|Y})^2$ & LCM and CCV & RCP or RCP$^2$\\ \midrule
    \boxwidth[l]{SAVE \ \cite{dennis2000save}}{DR \hfill \cite{li2007directional}}
    &
    $ \bbE( 2\Id - \Covv{Z-Z') | Y , Y' } )^2 $
    &
    LCM and CCV & RCP or RCP$^2$\\ \midrule
    \boxwidth[l]{SAVE \ \cite{dennis2000save}}{pHd \hfill \cite{li1992principal}} & $\bbE(Y - \bbE Y)(X - \bbE X)(X- \bbE X)^\top$ & normal $X$ & N/A
    \\
\bottomrule
\end{tabular}
}
\end{center}
\caption{
A summary of prominent inverse regression based methods (plus pHd).
We let $Z$ be the standardized $X$,
and $(Z',Y')$ an independent copy of $(Z,Y)$.
The table omits details on contour regression \cite{li2005contour} (strongly related to DR), iterative Hessian transformations \cite{cook2002dimension} (related to pHd),
and hybrid approaches \cite{zhu2007hybrid} (linear combinations of methods above).
}
\label{tab:inverse_regression}
\end{table}

As inverse regression based methods require only computation of finite sample means
and covariances, they are efficient and easy to implement. The matrix $\Lambda$
is usually approximated by partitioning the range $\Im(Y) = \cup_{\ell=1}^{J}\CR_{J,\ell}$,
and approximating statistics of $X|Y$ by empirical quantized statistics of $X|Y \in \CR_{J,\ell}$.
Therefore, only a single hyperparameter, the number of subsets $J$, needs to be tuned. A strategy
for choosing $J$ optimally is not known \cite{ma2013review}.

Nonparametric methods try to estimate the gradient field of the regression function $f$
based on the observation that the $d$ leading eigenvectors of $\bbE[\nabla f(X) \nabla f(X)^\top]$ (assuming $f$ is differentiable)
span the index space.
The concrete implementation of this idea differs between methods.
Popular examples are
minimum average variance estimation (MAVE),
outer product of gradient estimation (OPG),
and variants thereof \cite{xia2002adaptive}.
While MAVE converges
to the index space under mild assumptions,
it suffers from the curse of dimensionality due to nonparametric estimation of gradients of $f$.
The inverse MAVE (IMAVE) \cite{xia2002adaptive}
combines MAVE with inverse regression,
achieving $N^{-1/2}$-consistency under LCM.
Sliced regression \cite{SR} collects local MAVE estimates on response slices,
producing $N^{-1/2}$-consistent index estimates free of LCM for $ d \le 3 $.
Furthermore,
iterative generalizations of the average derivative estimation (ADE) \cite{ADE}
have been proved to be $N^{-1/2}$-consistent for $d \leq 3$ and $d\leq4$ \cite{dalalyan2008new,hristache2001structure}.

Compared to inverse regression methods, nonparametric methods rely on less stringent assumptions,
but are computationally more demanding, require more hyperparameter tuning, and are often more
complex to analyze. The relation between inverse regression and nonparametric methods has been
investigated in \cite{ma2012semiparametric,ma2013efficient} by introducing semiparametric methods.
The authors showed that
the computational efficiency
and simplicity of inverse regression methods come at the cost of assumptions such as LCM/CCV.
Moreover, they demonstrated that inverse regression methods
can be modified by including a nonparametric estimation step to achieve theoretical guarantees even
when LCM/CCV do not hold.

The work presented above originates from statistical literature, and, to the best of our
knowledge, focuses only on the index space estimation, completely omitting the subsequent regression step. This is different in the machine learning
community, where estimation of both $A$ and $f$
has been recently studied for the case $d=1$ \cite{ganti2017learning,
kakade2011efficient,kalai2009isotron,kereta2019nonlinear,kuchibhotla2016efficient,lanteri2020conditional,radchenko2015high}. The problem was also considered for $d\geq 1$ in
an active sampling setting, where the user is allowed to query data points $(X,Y)$ and the goal is to minimize the number of queries \cite{fornasier2012learning,
hemant2012active}. Moreover, model \eqref{eq:MIM} has strong ties
with shallow neural network models $f(x) = \sum_{i=1}^{m}g_i(a_i^\top x)$,
which are currently actively investigated
\cite{fornasier2018identification,janzamin2015beating,mondelli2018connection,soudry2016no}.

\subsection{Content and contributions of this work}
\label{subsec:content_and_contributions}
\paragraph*{Index space estimation}
We propose to estimate the index space
by the span of ordinary least squares solutions computed over level sets of the data (see Section \ref{sec:mim_algorithm} for details).
We call our method \emph{response-conditional least squares} (RCLS).
Our approach shares typical benefits of inverse regression based techniques: it is computationally
efficient, and easy to implement, as only a
single hyperparameter (number of level sets) needs to be specified. An additional advantage is
that ordinary least squares can be readily exchanged by variants leveraging
priors such as sparsity \cite{larsson2007linear,raskutti2011minimax} and further.

On the density level, we guarantee that RCLS finds a subspace $\Im(P_J)$ of $\Im(P)$
under the LCM assumption.
In the finite sample regime, we prove a concentration bound
\begin{equation} \label{eq:index_bound_intro}
\N{\MIMPJ{J} - P_J}_F \lesssim C(J) \sqrt{\frac{D}{N}} ,
\end{equation}
disentangling the influence of the number of samples $N$
and the number of level sets $J$ on the performance of our estimator (Corollary \ref{cor:mim_projection_error}).
Moreover, we show empirically that $C(J)$ in \eqref{eq:index_bound_intro}
tightly characterizes the influence of the hyperparameter on the estimator's performance, providing
guidance to how to choose it in practice.

\paragraph*{Link function regression}
We analyze the performance of kNN regression and piecewise polynomial regression
(with respect to a dyadic partition), when trained on the perturbed data set
$\{(\MIMP X_i, Y_i) : i = 1,\dots,N\}$ instead of $\{(\MIMPTrue X_i, Y_i): i = 1,\dots,N\}$,
where $\MIMP$ is \emph{any} estimate of the index projection $\MIMPTrue$.
Specifically, we prove for sub-Gaussian $X$, $(L,s)$-smooth $g$ (see Definition \ref{def:smoothness}), and almost surely bounded $f(X)$,
that the estimator $\hat f$ satisfies the generalization bound (Theorems \ref{thm:regression_error_kNN} and \ref{thm:regression_error_piecewise_pols})
\begin{equation}
\label{eq:generalization_bound_intro}
\bbE \left(\hat f(X) - f(X)\right)^2 \lesssim  N^{-\frac{2s}{2s+d}} + \N{\MIMP - \MIMPTrue}^{\min\{2s,2\}},
\end{equation}
where $s \in (0,1]$ in the case of kNN, and $s \in (0,+\infty)$ in the case of piecewise polynomials.
The bound \eqref{eq:generalization_bound_intro}
shows that optimal estimation rates (in the minimax sense) are achieved by traditional regressors for $d=1$ and $s > \frac{1}{2}$,
or $d \geq 2$ and any $ s > 0 $, provided $\Vert \MIMP - \MIMPTrue\Vert \in \CO(N^{-1/2})$.
In particular, combining \eqref{eq:index_bound_intro} and \eqref{eq:generalization_bound_intro}
we obtain that RCLS paired with piecewise polynomial regression produces an optimal estimation of the multi-index model.

\subsection{Organization of the paper}
\label{subsec:organization}
Section \ref{sec:mim_algorithm} describes RCLS for index space estimation.
Section \ref{sec:mim_index_space_estimation} presents theoretical guarantees
on the population level and in the finite sample regime. Section \ref{sec:mim_regression}
establishes the generalization bound \eqref{eq:generalization_bound_intro}.
Section \ref{sec:mim_numerical_examples} compares RCLS with state-of-the-art methods on synthetic and real data sets.

\subsection{General notation}
\label{subsec:notation}
We let $\bbN_0$ be the set of natural numbers including $0$,
and $[J] := \{1,\ldots,J\}$.
We write $ a \vee b := \max\{a,b\}$ and $a \wedge b := \min\{a,b\}$.
Throughout the paper, $C$ stands for a universal constant that may change on each appearance.
We use $\Vert \cdot \Vert$ for the Euclidean norm of
vectors,
and $\N{\cdot}$, $\N{\cdot}_F$ for the spectral and Frobenius matrix norms, respectively.
For a symmetric real matrix $A \in \bbR^{D\times D}$,
we denote the ordered eigenvalues as $\lambda_1(A) \geq \cdots \geq \lambda_D(A)$
and the corresponding eigenvectors as $u_1(A),\dots,u_D(A)$.

We denote expectation and covariance of a random vector $X$ by $\bbE [X]$ and $\Covv{X}$, respectively,
and let $\centerRV{X}: = X - \bbE [X]$.
The sub-Gaussian norm of a random variable $Z$ is
$\N{Z}_{\psi_2} := \inf\{t > 0 : \bbE\exp(Z^2/t^2) \leq 2\}$.
Similarly, the sub-Exponential norm is
$\N{Z}_{\psi_1}=\inf\{t > 0 : \bbE\exp(\SN{Z}/t) \leq 2\}$.
Finally, we abbreviate the mean squared error of an estimator $\smash{\hat f}$ of $f$ by $ \textrm{MSE}( \smash{\hat f} , f ) := \bbE | \smash{\hat f(X)} - f(X) |^2 $.

\section{Index space estimation by response-conditional least squares}
\label{sec:mim_algorithm}

\begin{algorithm}[b]
\caption{Index space estimation via RCLS}
\label{alg:index_space_estimation}
\begin{algorithmic}
  \REQUIRE Parameters $J$ and $\tilde d$, data set $\{(X_i,Y_i) : i \in [N]\}$
  \STATE Split data into $\{\CX_{J,\ell} : \ell \in [J]\}$ and $\{\CY_{J,\ell} : \ell \in [J]\}$ according to \eqref{eq:level_set_splitting}
  \FOR{$\ell=1,\ldots,J$}
    \STATE $\hat{b}_{J,\ell} :=  \hat \Sigma_{J,\ell}^{\dagger}\,\lsmean{(\CX_{J,\ell}, \CY_{J,\ell})}{(X - \lsmean{\CX_{J,\ell}}{X})(Y-\lsmean{\CY_{J,\ell}}{Y})}$
    \STATE $\hat \rho_{J,\ell} := \SN{\CX_{J,\ell}}N^{-1}$
  \ENDFOR
\STATE $\hat M_J := \sum_{\ell=1}^{J}  \hat \rho_{J,\ell} \hat b_{J,\ell}\hat b_{J,\ell}^\top$
\ENSURE $\MIMPJ{J}(\tilde d)$, the orthoprojector onto the leading $\tilde d$ eigenvectors of $\hat M_J$
\end{algorithmic}
\end{algorithm}

\noindent
We first describe response-conditional least squares (RCLS) and then highlight
advantages and disadvantages of the approach compared to
other methods in the literature (see Section \ref{subsec:overview_index_space_estimation}).

\paragraph*{RCLS} For the sake of simplicity we assume here that $\Im(Y)$ is bounded. First, let  $\Im(Y) = \cup_{\ell = 1}^{J} \CR_{J,\ell}$
be a dyadic decomposition of the range into $J$ intervals. For example,
this means  $\CR_{J,\ell} := [\frac{\ell-1}{J}, \frac{\ell}{J})$ in the case $\Im(Y) = [0,1)$.
We assign the samples $\{(X_i,Y_i) : i \in [N]\}$ to subsets
\begin{align}
\label{eq:level_set_splitting}
\CY_{J,\ell} := \{Y_i : Y_i \in \CR_{J,\ell}\},\quad \textrm{and}\quad \CX_{J,\ell} := \{X_i : Y_i \in \CR_{J,\ell}\},
\end{align}
which we refer to as level sets in the following.
On each level set we solve an ordinary least squares problem.
That is, for the empirical conditional covariance matrix $\hat\Sigma_{J,\ell} :=\lsmean{\CX_j}(X-\lsmean{\CX_{J,\ell}}{X})(X-\lsmean{\CX_{J,\ell}}{X})^\top$,
where $\lsmean{\CX_{J,\ell}}{X}$ is the usual finite sample mean of $X$ over $\CX_{J,\ell}$,
we compute vectors
\begin{align}
\label{eq:def_sample_b}
\hat b_{J,\ell} := \hat \Sigma_{J,\ell}^{\dagger} \lsmean{(\CX_{J,\ell},\CY_{J,\ell})}{(X - \lsmean{\CX_{J,\ell}}{X})(Y - \lsmean{\CY_{J\ell}}{Y}}).
\end{align}
Intuitively speaking, $\hat b_{J,\ell}$ can be seen as an estimate of the averaged
gradient of the regression function $f$ over the level set $\CR_{J,\ell}$. Taking into
account the model $f(x) = g(A^\top x)$, $\hat b_{J,\ell}$ should therefore become increasingly close
to the index space $\Im(A)$ as the number of samples in $\CX_{J,\ell}$ increases.
This motivates to approximate the index space by the leading eigenvectors of an outer product matrix
of vectors $\{\hat b_{J,\ell} : \ell \in [J]\}$. We set $\hat \rho_{J,\ell} := \SN{\CX_{J,\ell}}N^{-1}$ and then compute
\begin{align}
\label{eq:def_sample_M_J}
\MIMPJ{J}(\tilde d) := \sum\limits_{i=1}^{\tilde d}u_i(\hat M_J) u_i(\hat M_J)^\top,
\quad \textrm{where}\quad \hat M_J := \sum_{\ell=1}^{J}  \hat \rho_{J,\ell} \hat b_{J,\ell}\hat b_{J,\ell}^\top.
\end{align}
The parameter $\tilde d \leq d$ is user-specified and ideally equals $\dim(\Span\{\hat b_{J,\ell} : \ell \in [J]\})$ in the limit
$N\rightarrow \infty$. If this value is unknown, we select it via model selection techniques
or by inspecting the spectrum of $\hat M_J$. The procedure is summarized in Algorithm \ref{alg:index_space_estimation}.

\begin{remark}[Choice of partition]
\label{rem:choice_of_partition}
One possible way to decompose $\Im(Y)$ is by dyadic cells.
However, the analysis in Section \ref{sec:mim_index_space_estimation}
does not require the dyadic structure and can instead be conducted with arbitrary partitions.
\end{remark}

\begin{remark}[Algorithmic complexity]
\label{rem:algorithmic_complexity}
The main computational demand is constructing the vectors $\{\hat b_{J,\ell} : \ell \in [J]\}$. Assuming we use a partition of disjoint level sets $\CR_{J,\ell}$, i.e. each sample is only used once in the construction of $\hat M_J$, the cost for this is $\CO(\sum_{\ell=1}^{J}\SN{\CX_{J,\ell}}D^2) = \CO(ND^2)$.
\end{remark}

\paragraph*{Comparison of RCLS with inverse regression methods}
In RCLS, response conditioning serves to localize and produce multiple estimates
rather than induce isotropy in the marginal distribution
(\emph{e.g.} no conditioning is required in the single-index case);
hence, it is not a typical inverse regression method.
At the same time, it shares the same
general advantages: it is simple, computationally efficient, and provable.
While in inverse regression no optimal slicing procedure is known,
RCLS admits a tight parameter characterization which allows for a straight optimization.
On par with all inverse regression methods, RCLS requires the LCM assumption.
Although it is often more or at least as accurate as second order inverse regression methods,
such as CR and DR,
it does not need the CCV assumption.
This is a major generalization since,
as pointed out in \cite{ma2013review},
assuming both LCM and CCV for all directions reduces $X$ to the normal distribution.
RCLS low computational cost matches that of typical inverse regression estimates (except CR, which is $\CO(N^2D^2)$).

\paragraph*{Comparison of RCLS with nonparametric methods}
Essentially relying on gradient field estimation,
RCLS has strong ties with nonparametric methods,
but it has lower computation cost and it is easier to implement.
Note that nonparametric methods typically involve kernel smoothing,
leading to complexities quadratic in the sample size $N$.
Such costs are linearizable resorting for example to nearest neighbor truncation,
but while naive kNN still requires the computation of $\CO(N^2)$ distances,
hierarchical structures for fast neighbor search, such as k-d and cover trees \cite{kd-tree,cover-tree}, imply constants exponential in the dimension $D$,
not to mention the overhead resulting from cross-validating the number of neighbors.
Cross-validation is in principle also required for bandwidth selection,
even for joint tuning of two different bandwidths \cite{SR},
since optimal choices beyond rules of thumb (\emph{e.g.} the ``normal reference'') are to date an open problem.
Last but not least, kernel estimates are sensitive to the curse of dimensionality,
whose overcoming requires further complications, algorithmic tweaks, initializations and iterative procedures \cite{SR,xia2002adaptive}.

\section{Guarantees for RCLS}
\label{sec:mim_index_space_estimation}
We introduce population counterparts of $\hat b_{J,\ell}$ and $\hat M_J$ given by
\begin{align*}
b_{J,\ell} := \Sigma_{J,\ell}^{\dagger}\Covv{X,Y|Y \in \CR_{J,\ell}}\quad \textrm{ and }\quad M_J := \sum_{\ell=1}^{J}\rho_{J,\ell} b_{J,\ell} b_{J,\ell}^\top,
\end{align*}
where $\Sigma_{J,\ell} := \Covv{X|Y \in \CR_{J,l}}$, and $\rho_{J,\ell} := \bbP(Y \in \CR_{J,\ell})$.
All quantities thus far are defined through the random vector $(X,Y)$ without using the regression function.
In fact, in this section we can technically avoid specifying the regression function by
assuming a more general setting, where $\Im(\MIMPTrue)$ is defined as the minimal dimensional subspace with
\begin{enumerate}[label=(A\arabic*)]
\item\label{ass_MIM:independence} $Y \independent X|\PVX$.
\end{enumerate}
As mentioned in Section \ref{subsec:overview_index_space_estimation}, \ref{ass_MIM:independence}
uniquely defines $\Im(P)$ except for degenerate cases, which we exclude here.
Moreover, \ref{ass_MIM:independence} with $P=AA^\top$ generalizes \eqref{eq:MIM} and $\zeta \independent X|A^\top X$.

In the following analysis, we also require the following assumptions.
\begin{enumerate}[label=(A\arabic*)]
\setcounter{enumi}{1}
\item\label{ass_MIM:linear_conditional_mean}
$\bbE[X|\PVX] = \PVX$ almost surely;
\item\label{ass_MIM:subGauss} $X$ and $Y$ are sub-Gaussian random variables.
\end{enumerate}

\noindent
\ref{ass_MIM:linear_conditional_mean} is the LCM assumption introduced in Section \ref{subsec:overview_index_space_estimation}
and is required in all inverse regression based techniques like SIR, SAVE or DR. It is satisfied for example
for any elliptical distribution and ensures $\Im(M_J)\subseteq \Im(A)$ as shown in Proposition \ref{prop:MIM_pop_level_prop} below.
\ref{ass_MIM:subGauss} is maximally general to use the tools developed in the framework of sub-Gaussian random variables,
namely finite sample concentration bounds. Examples of sub-Gaussian random variables include bounded distributions,
the normal distribution, or more generally random variables for which all one-dimensional marginals have tails
that exhibit a Gaussian-like decay after a certain threshold \cite{vershynin2018high}.

\subsection{Population level}
The population level results are summarized in the following proposition.
\label{subsec:index_space_pop_level}
\begin{proposition}
\label{prop:MIM_pop_level_prop}
If $(X,Y)$ satisfies  \ref{ass_MIM:independence} and \ref{ass_MIM:linear_conditional_mean},
then $b_{J,\ell} \in \Im(A)$ for any $\ell \in [J]$ and any $J$. Also, $\Im(M_J)\subseteq \Im(A)$, with equality if
$\lambda_d\left(M_{J}\right) > 0$.
\end{proposition}
\noindent
We need the following result for the proof of Proposition \ref{prop:MIM_pop_level_prop}.
\begin{lemma}
\label{lem:E-CovX|Y}
Under \ref{ass_MIM:independence} and \ref{ass_MIM:linear_conditional_mean}, we get
\begin{enumerate}[label=\textnormal{(\alph*)}]
\item \label{it:EX|Y} $ \bbE [X | Y ]  = \bbE[\PVX|Y]$ almost surely, or equivalently $ \bbE[\QVX|Y] = 0 $;
\item \label{it:CovX|Y} $ \Covv{ X \mid Y} = \Covv{ \PVX \mid Y} + \Covv{ \QVX \mid Y} $ almost surely.
\end{enumerate}
\end{lemma}
\begin{proof}
\ref{it:EX|Y}.
\sloppy The towering property of conditional expectations yields $ \bbE [ X | Y ] =  \bbE [ \bbE [ X | \PVX , Y ] | Y ] $.
Assumption \ref{ass_MIM:independence} implies $Y \independent X | \PVX$,
and thus $\bbE [ X | \PVX , Y ] = \bbE [ X | \PVX ] = \PVX$
by assumption \ref{ass_MIM:linear_conditional_mean}.

\noindent
\ref{it:CovX|Y}.
By the law of total covariance,
\begin{align*}
\Covv{\PVX, \QVX | Y} & = \bbE[\Covv{\PVX, \QVX | \PVX, Y} | Y] \\
& +\Covv{\bbE[\QVX | \PVX, Y], \bbE[\PVX | \PVX, Y] | Y}
= 0 + 0,
\end{align*}
where we used $\bbE [ X | \PVX , Y ] = \PVX$ as shown in the proof of \ref{it:EX|Y}.
Therefore,
\begin{equation*}
\Covv{X \mid Y} = \Covv{ \PVX + \QVX \mid Y} =
\Covv{\PVX \mid Y} +  \Covv{\QVX \mid Y} . \qedhere
\end{equation*}
\end{proof}
\begin{proof}[Proof of Proposition \ref{prop:MIM_pop_level_prop}]
We only show that $b_{J,\ell} \in \Im(A)$ for all $R_{J,\ell}$, since $\Im(M_J)\subseteq \Im(A)$ follows
immediately.
We have
\begin{align*}
\Covv{\QVX,Y|Y \in \CR_{J,\ell}} &= \bbE\left[\Covv{\QVX,Y|Y} |Y \in \CR_{J,\ell}\right] \\
&+\Covv{\bbE[\QVX |Y ], \bbE[Y | Y ] | Y \in \CR_{J,\ell}}\\
&=0 + \Covv{\bbE[\QVX | Y], Y | Y \in \CR_{J,\ell}} = 0,
\end{align*}
where the last equality follows from $\bbE[\QVX|Y] = 0$ by Lemma \ref{lem:E-CovX|Y}. Therefore
\begin{align*}
\Covv{X,Y|Y \in \CR_{J,\ell}} &= \Covv{\PVX,Y|Y \in \CR_{J,\ell}} + \Covv{\QVX,Y|Y \in \CR_{J,\ell}}\\
&= \Covv{\PVX,Y|Y \in \CR_{J,\ell}} + 0 .
\end{align*}
Furthermore, statement (b) of Lemma \ref{lem:E-CovX|Y} implies
$$ \Sigma_{J,\ell} = \Covv{\PVX | Y \in \CR_{J,\ell}} + \Covv{\QVX | Y \in \CR_{J,\ell}} , $$
\sloppy hence the eigenspace of $\Sigma_{J,\ell}$ decomposes orthogonally into eigenspaces of
$\Covv{\PVX | Y \in \CR_{J,\ell}}$ and of $\Covv{\QVX | Y \in \CR_{J,\ell}}$.
The same holds for $\Sigma_{J,\ell}^{\dagger}$ because the eigenvectors are precisely the same
as for $\Sigma_{J,\ell}$. This implies $\Sigma_{J,\ell}^{\dagger}z \in \Im(\MIMPTrue)$ for all $z \in \Im(\MIMPTrue)$,
and the result follows by
\begin{equation*}
b_{J,\ell} = \Sigma_{J,\ell}^{\dagger}\Covv{X,Y|Y \in \CR_{J,\ell}} = \Sigma_{J,\ell}^{\dagger}\Covv{\PVX,Y|Y \in \CR_{J,\ell}} \in A. \qedhere
\end{equation*}
\end{proof}
\paragraph*{Exhaustiveness}
Proposition \ref{prop:MIM_pop_level_prop} ensures exhaustiveness of RCLS (on the population level)
whenever $d$ out of the $J$ least squares vectors $b_{J,\ell}$ are linearly independent.
Even when this is not the case,
we believe that RCLS generically finds a subspace of the index space that accounts for most of the variability in $f$,
thereby allowing for a \emph{sufficient} dimension reduction.
The rationale behind this is that the $b_{J,\ell}$'s can be interpreted as averaged gradients over approximate level sets,
and thus they provide samples of the first order behavior of $f$
along the chosen partition.
This claim is supported numerically in Section \ref{subsec:real_data_experiments},
where RCLS performs better or as good as all inverse regression based methods listed in Table \ref{tab:inverse_regression}.

Analyzing the exhaustiveness of inverse regression estimators is challenging
since in general it is easy to construct examples
where some directions of the index space only show up in the tails of $(X,Y)$.
This also justifies why most typical exhaustiveness conditions
such as RCP and RCP$^2$ are formulated on the nonquantized level,
and therefore do not quite imply exhaustiveness of the actual quantized estimator.
The only exception we are aware of is \cite[Theorem 3.1]{li2005contour},
where sufficient conditions for the exhaustiveness of the estimator are provided by
decoupling the roles of regression function and noise.

Lastly, we mention that it is possible to further enrich the space $\Im(M_J)$
by adding outer products of vectors
$B b_{J,\ell}$ for matrices $B$ which map $\Im(P)$
to $\Im(P)$. This resembles the idea behind the iHt method \cite{cook2004determining},
where $B$ is chosen as a positive power of the average residual- or response-based
Hessian matrix \cite{cook2004determining,li1992principal}. Other choices are powers
of $\Sigma_{J,\ell}$ or $\Sigma_{J,\ell}^{\dagger}$, which map $\Im(P)$ to $\Im(P)$
under \ref{ass_MIM:independence} and \ref{ass_MIM:linear_conditional_mean}.

\subsection{Finite sample guarantees}
\label{sec:mim_index_space_finite_sample}
We now analyze the finite sample performance of $\MIMPJ{J}(\tilde d)$
as an estimator for the orthoprojector $P_J$ satisfying $\Im(P_J) = \Span\{b_{J,\ell} : \ell \in [J]\}$.
Our main result is a $N^{-1/2}$ convergence rate, which is typically also achieved
for inverse regression based methods. Additionally however, we carefully track
the influence of the induced level set partition on the estimator's
performance in order to understand the influence of the hyperparameter $J$. To achieve this,
we rely on an anisotropic bound for the concentration of individual
ordinary least squares vector $\hat b_{J,\ell}$ around $b_{J,\ell}$. The resulting error
bound links the accuracy of RCLS to the geometry of $X|Y \in \CR_{J,\ell}$ encoded
in spectral properties of conditional covariance matrices $\Sigma_{J,\ell} = \Covv{X|Y \in \CR_{J,\ell}}$.

Before we begin the analysis, we introduce some notation. Let $X|_{J,\ell}$ and $Y|_{J,\ell}$
denote random variables $X$ and $Y$ conditioned on $Y \in \CR_{J,\ell}$. By \ref{ass_MIM:subGauss} and Lemma \ref{lem:conditional_subgaussianity},
$X|_{J,\ell}$ and $Y|_{J,\ell}$ are sub-Gaussian whenever $\rho_{J,\ell} > 0$, which implies
that $\Vert X|_{J,\ell}\Vert_{\psi_2}$ and $\Vert Y|_{J,\ell}\Vert_{\psi_2}$ are finite.
Moreover, we define $P_{J,\ell}$ as the orthoprojector onto $\Span\{b_{J,\ell}\}$ and $Q_{J,\ell}:=\Id - P_{J,\ell}$.

\subsubsection{Anisotropic concentration}
\label{subsubsec:anisotropic_concentration_linreg}
An anisotropic concentration bound for $\hat b_{J,\ell} - b_{J,\ell}$
uses the orthogonal decomposition
\begin{equation}
\label{eq:mim_linreg_error_ortho_decomp}
\hat b_{J,\ell} - b_{J,\ell} = P_{J,\ell}(\hat b_{J,\ell} - b_{J,\ell}) + Q_{J,\ell}(\hat b_{J,\ell} - b_{J,\ell}) =
P_{J,\ell}(\hat b_{J,\ell} - b_{J,\ell}) + Q_{J,\ell}\hat b_{J,\ell},
\end{equation}
and finds separate bounds for the terms $P_{J,\ell}(\hat b_{J,\ell} - b_{J,\ell})$ and $Q_{J,\ell}\hat b_{J,\ell}$.
To see why those terms play different roles when estimating $\Im(\MIMPTrue)$,
let us consider the illustrative case of the single-index model, where $P=aa^\top$
for some $a \in \bbS^{D-1}$. We can estimate $a$ by the direction of any
$b_{J,\ell}$, because any nonzero $b_{J,\ell}$ is aligned with $a$
under \ref{ass_MIM:independence} and \ref{ass_MIM:linear_conditional_mean}.
Using few algebraic manipulations we have, with $Q := \Id - P$,
\begin{equation}
\label{eq:sim_example_direction}
\N{\frac{\hat b_{J,\ell}}{\N{\hat b_{J,\ell}}} - a} = \N{\frac{\hat b_{J,\ell}}{\N{\hat b_{J,\ell}}} - \frac{b_{J,\ell}}{\N{b_{J,\ell}}}} \leq \frac{\N{Q\hat b_{J,\ell}}}{\N{b_{J,\ell}}-\N{P(\hat b_{J,\ell} - b_{J,\ell})}},
\end{equation}
whenever $\hat b_{J,\ell}^\top b_{J,\ell} > 0$.
This reveals that the error is dominated by $\Vert Q\hat b_{J,\ell}\Vert$, whereas $\Vert P(\hat b_{J,\ell} - b_{J,\ell})\Vert$
is a higher order error term as soon as  $\SN{\CX_{J,\ell}}$ is sufficiently large.
A similar observation will be established for higher dimensional index spaces in \eqref{eq:bound_after_davis_kahan}
below.

Anisotropic concentration bounds for ordinary least squares vectors have been recently
provided in \cite{acapkarate}.
To restate the bounds, we introduce a directional sub-Gaussian condition number
$\kappa_{J,\ell} := \kappa(P_{J,\ell},X|_{J,\ell}) \vee \kappa(Q_{J,\ell}, X|_{J,\ell})$, where
(recall $\centerRV{Z} = Z - \bbE Z$)
\begin{align}
\label{eq:mim_def_kappa}
\kappa(L, X) &:= \N{L\Covv{X}^{\dagger} \centerRV{X}}_{\psi_2}^2 \N{L\centerRV{X}}_{\psi_2}^2.
\end{align}
As described in \cite{acapkarate}, $\kappa(L, X)$ is related to the restricted matrix condition number defined by
$\tilde\kappa(L,\Covv{X}) := \Vert L\Covv{X}^{\dagger}L\Vert\N{L\Covv{X} L}$,
which measures the heterogeneity of eigenvalues of $\Covv{X}$, when restricting the eigenspaces to $\Im(L)$.
In fact, if $X$ follows a normal distribution, the sub-Gaussian norm is a tight variance
proxy and $\kappa(L,X)$ differs from $\tilde \kappa(L,\Covv{X})$ by a constant factor that only
depends on the precise definition of the sub-Gaussian norm.

We further introduce the standardized random variable $\centerRV{Z}|_{J,\ell} := \Sigma_{J,\ell}^{-1/2}\centerRV{X}|_{J,\ell}$,
where $\Sigma_{J,\ell}^{-1/2}$ is the matrix square root of
$\Sigma_{J,\ell}^{\dagger}$. As a consequence of the standardization, we have
$\textrm{Cov}(\centerRV{Z}|_{J,\ell}) = \Id_D$.

\begin{lemma}[Anisotropic ordinary least squares bounds]
\label{lem:linear_regression_bounds}
Let $J \in \bbN$, $\ell \in [J]$ and assume \ref{ass_MIM:subGauss}.
For fixed $u > 0$, $\varepsilon > 0$, with $\SN{\CX_{J,\ell}} > C(D+u)(\Vert\centerRV{Z}|_{J,\ell}\Vert_{\psi_2}^4 \vee \varepsilon^{-2})$, we get,
with probability at least $1-\exp(-u)$,
\begin{align}
\label{eq:lin_reg_P}
&\N{P_{J,\ell}(b_{J,\ell} - \hat b_{J,\ell})} \leq \varepsilon \sqrt{\kappa_{J,\ell}}
\N{\centerRV{Y}|_{J,\ell}}_{\psi_2} \N{P_{J,\ell} \Sigma_{J,\ell}^{\dagger}\centerRV{X}|_{J,\ell}}_{\psi_2},\\
\label{eq:lin_reg_Q}
&\N{Q_{J,\ell}\hat b_{J,\ell}} \leq \varepsilon \sqrt{\kappa_{J,\ell}} \N{\centerRV{Y}|_{J,\ell}}_{\psi_2} \N{Q_{J,\ell} \Sigma_{J,\ell}^{\dagger}\centerRV{X}|_{J,\ell}}_{\psi_2} ,
\end{align}
Furthermore, we have
\begin{equation}
\label{eq:b_J_ell_bound}
\N{b_{J,\ell}} \leq 2\N{\centerRV{Y}|_{J,\ell}}_{\psi_2} \N{P_{J,\ell} \Sigma_{J,\ell}^{\dagger}\centerRV{X}|_{J,\ell}}_{\psi_2}.
\end{equation}
\end{lemma}
\begin{proof}
The concentration bounds are precisely \cite[Lemma 4.1]{acapkarate} adjusted to the notation used here.
\eqref{eq:b_J_ell_bound} follows from $b_{J,\ell} = P_{J,\ell} \Sigma_{J,\ell}^{\dagger}\Covv{X|_{J,\ell}, Y|_{J,\ell}}$ and \cite[Lemma 6.5]{acapkarate}.
\end{proof}
\noindent

\begin{figure}[h]
        \centering
        \includegraphics[trim={0 1.5cm 0 1.cm}, clip, width=0.8\textwidth]{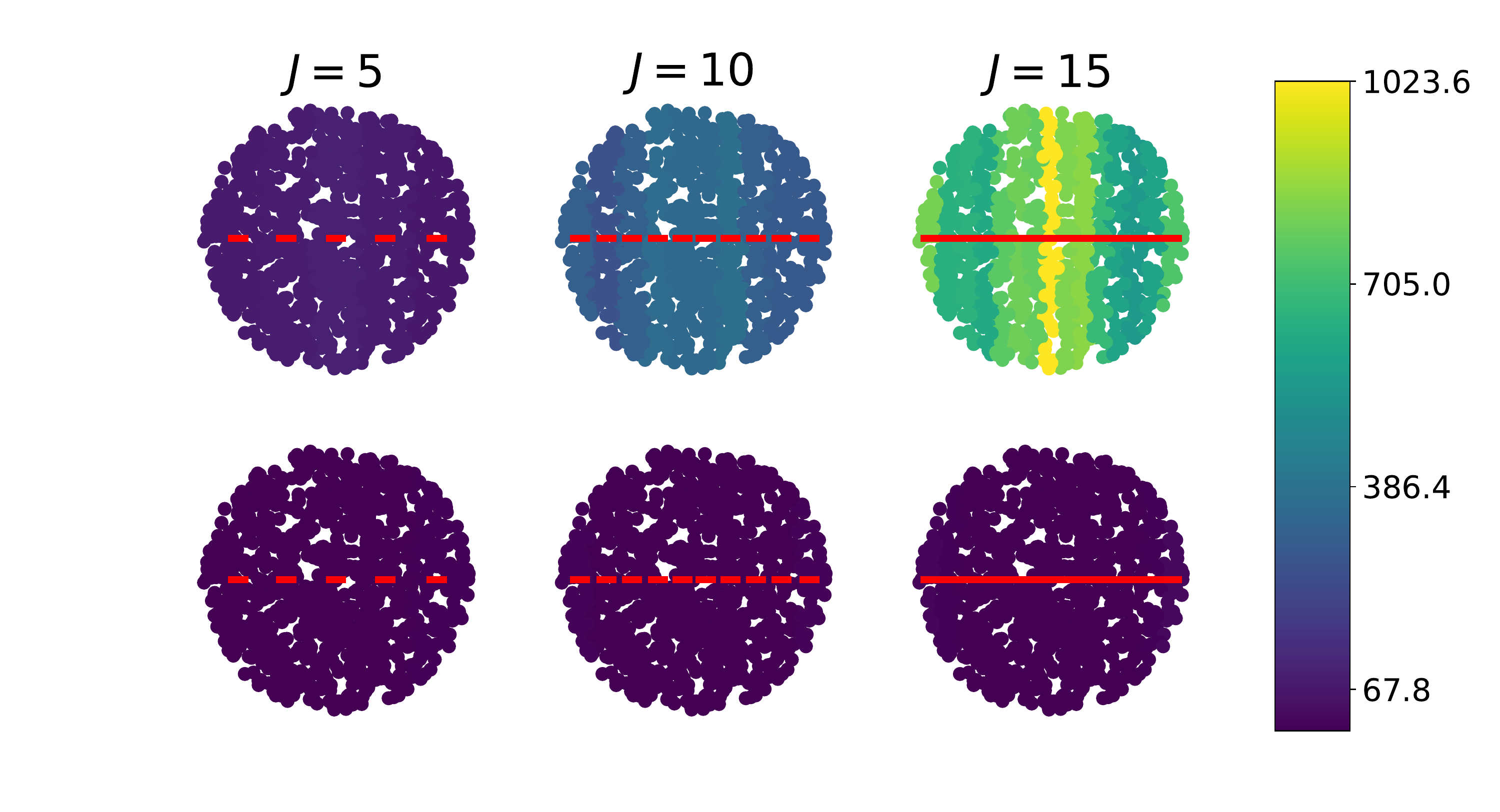}
        \caption{We sample $X \sim \textrm{Uni}(\{X : \N{X} \leq 1\})$ and $Y = 1/(1+\exp(-X_1))$, and show
        $\Vert P \Sigma_{J,\ell}^{\dagger} P\Vert$ (top) and $\Vert Q \Sigma_{J,\ell}^{\dagger}Q\Vert$ (bottom) row for $P =e_1e_1^\top$, $Q = \Id - P$,
        and all $\ell \in [J]$ for $J \in \{5,10,15\}$. The red lines mark local ordinary least squares vectors.
        We see that $\Vert P \Sigma_{J,\ell}^{\dagger} P\Vert$ increases substantially when increasing the number
        of level sets $J$, while $\Vert Q \Sigma_{J,\ell}^{\dagger}Q\Vert$ remains roughly constant.}
        \label{fig:levelset_details}
\end{figure}

\noindent
Equations \eqref{eq:lin_reg_P} and \eqref{eq:lin_reg_Q} in Lemma \ref{lem:linear_regression_bounds} reveal that
the concentration of $P_{J,\ell}(b_{J,\ell} - \hat b_{J,\ell})$ and $Q_{J,\ell}\hat b_{J,\ell}$  scale with sub-Gaussian norms of
$P_{J,\ell}\Sigma^{\dagger}_{J,\ell}\centerRV{X}$, and $Q_{J,\ell}\Sigma^{\dagger}_{J,\ell}\centerRV{X}$
(we can intuitively think of $\Vert P_{J,\ell}\Sigma_{J,\ell}^{\dagger}P_{J,\ell}\Vert$,
and $\Vert Q_{J,\ell}\Sigma_{J,\ell}^{\dagger}Q_{J,\ell}\Vert$). In many scenarios, both norms, if
viewed as functions of the parameter $J$, behave very differently.
This is because increasing the number of level sets $J$ typically reduces
the variance in the direction of the least squares solution $b_{J,\ell}$,
and therefore increases $\Vert P_{J,\ell}\Sigma_{J,\ell}^{\dagger}P_{J,\ell}\Vert$, while
$\Vert Q_{J,\ell}\Sigma_{J,\ell}^{\dagger}Q_{J,\ell}\Vert$ is often not affected.
The effect is particularly strong for single-index models with
monotone link functions, as illustrated in Figure \ref{fig:levelset_details},
but it can also be observed in more general scenarios, for instance if $f$ follows
a monotone single-index model locally on one of the level sets. Recalling
\eqref{eq:sim_example_direction}, using anisotropic concentration is therefore necessary, if we aim at an
accurate description of the projection error in terms of both, $N$ and $J$.

To simplify notation in the following, we introduce the shorthands
\[
\eta_{J,\ell}^{\parallel} := \N{\centerRV{Y}|_{J,\ell}}_{\psi_2} \N{P_{J,\ell} \Sigma_{J,\ell}^{\dagger}\centerRV{X}|_{J,\ell}}_{\psi_2},\quad
\eta_{J,\ell}^{\perp} := \N{\centerRV{Y}|_{J,\ell}}_{\psi_2} \N{Q_{J,\ell} \Sigma_{J,\ell}^{\dagger}\centerRV{X}|_{J,\ell}}_{\psi_2},
\]
and $\eta_{J,\ell} := \eta_{J,\ell}^{\parallel} + \eta_{J,\ell}^{\perp}$.

\subsubsection{Concentration bounds for index space estimation}
\label{subsubsec:concentration_bounds_index_space_estimation}
Our goal is now to provide concentration bounds for $\MIMPJ{J}(\tilde d)$
around $\MIMPTrue_J$.
Using the Davis-Kahan Theorem \cite[Theorem 7.3.1]{bhatia2013matrix}
we have for $Q_J := \Id - \MIMPTrue_J$
\begin{align}
\label{eq:bound_after_davis_kahan}
\N{\MIMPJ{J}(\tilde d)Q_J\!}_F
\!\leq\! \frac{\N{\MIMPJ{J}(\tilde d)\left(\hat M_J - M_J\right)Q_J}_F}{\lambda_{\tilde d}(\hat M_J)}
\!\leq\! \frac{\N{\left(\hat M_J - M_J\right)Q_J}_F}{\lambda_{\tilde d}( M_J )\!-\!\N{\hat M_J\!-\!M_J}},
\end{align}
where we used Weyl's bound \cite{weyl1912asymptotische} to get
$\lambda_{\tilde d}(\hat M_J)  \geq \lambda_{\tilde d}( M_J ) - \textN{\hat M_J - M_J}$ in the second inequality. It remains
to develop concentration bounds for $\Vert (\hat M_J - M_J)Q_J\Vert_F$, which dictates the projection error,
and $\Vert \hat M_J - M_J \Vert$ to ensure that the denominator does not vanish.

\begin{theorem}
\label{thm:M_J_parallel_bound}
Let \ref{ass_MIM:independence} - \ref{ass_MIM:subGauss} hold. Fix $u > 0$, $\varepsilon > 0$, and define
$$
\maxStandardNorm := \max_{\ell \in [J]}\Vert \centerRV{Z}|_{J,\ell}\Vert_{\psi_2}^4, \quad \minRho := \min_{\ell \in [J]}\rho_{J,\ell} .
$$
Whenever $N > C(D + u + \log(J))(\maxStandardNorm\ \minRho^{-1} \vee \varepsilon^{-2})$ we have
\begin{align}
\label{eq:M_J_parallel_bound}
&\bbP\left(\N{\hat M_J - M_J}_F \leq \varepsilon \sum_{\ell=1}^{J}\sqrt{\kappa_{J,\ell}}\eta_{J,\ell}^2\right) \geq 1 - \exp(-u).
\end{align}
\end{theorem}

\begin{proof}
The first step is to decompose the error according to
\begin{align*}
&\N{M_J - \hat M_J}_F =\N{\sum_{\ell=1}^{J} \hat \rho_{J,\ell} \left(\hat b_{J,\ell} \hat b_{J,\ell}^\top - b_{J,\ell} b_{J,\ell}^\top\right)- \left( \rho_{J,\ell} - \hat\rho_{J,\ell} \right)  b_{J,\ell} b_{J,\ell}^\top}_F\\
&\quad\leq\sum_{\ell=1}^{J}\hat \rho_{J,\ell} \N{\hat b_{J,\ell} \hat b_{J,\ell}^\top - b_{J,\ell}  b_{J,\ell}^\top}_F + \sum_{\ell=1}^{J} \SN{\hat \rho_{J,\ell} -  \rho_{J,\ell}}
\N{b_{J,\ell} b_{J,\ell}^\top}_F\\
&\quad\leq\underbrace{\sum_{\ell=1}^{J} \hat \rho_{J,\ell} \left(\N{\hat b_{J,\ell} - b_{J,\ell}} + 2\N{b_{J,\ell}}\right)\N{\hat b_{J,\ell} -  b_{J,\ell}}}_{=: T_1} +
\underbrace{\sum_{\ell=1}^{J} \SN{\hat \rho_{J,\ell} -  \rho_{J,\ell}}\N{b_{J,\ell}}^2}_{=:T_2}.
\end{align*}
The second term can be bounded using Lemma \ref{lem:linear_regression_bounds} and \ref{lem:bound_rho}.
Specifically, Lemma \ref{lem:linear_regression_bounds} implies $\N{b_{J,\ell}} \leq 2\eta_{J,\ell}$,
and \eqref{eq:emp_density_hoeffding} in Lemma \ref{lem:bound_rho} with a union bound argument over $\ell \in [J]$ shows
\begin{align}
\label{eq:conditioning_for_rho_1}
&\bbP\left(\forall \ell \in [J]:\SN{\rho_{J,\ell} - \hat \rho_{J,\ell}} \le \varepsilon \right) \geq 1-\exp(-u),
\end{align}
provided $N >  C(u+\log(J))\varepsilon^{-2}$.
Thus we have $T_2 \leq 4 \varepsilon \sum_{\ell \in [J]}\eta_{J,\ell}^2$ with probability $1-\exp(-u)$ whenever $N >  C(u+\log(J))\varepsilon^{-2}$.

\noindent
To bound $T_1$ we first need to ensure that each level set is sufficiently populated.
Equation \eqref{eq:emp_density_chernoff} in Lemma \ref{lem:bound_rho}, and the union bound over $\ell \in [J]$ give
\begin{align}
\label{eq:conditioning_for_rho_2}
\bbP \left(\forall \ell \in [J]: \SN{\hat \rho_{J,\ell} - \rho_{J,\ell}} \leq \frac{1}{2} \rho_{J,\ell} \right) &\geq 1 - 2\exp(-u)
\end{align}
provided $N > C\frac{u + \log(J)}{\minRho}$.
It follows that $1/2 \leq \hat \rho_{J,\ell}/\rho_{J,\ell} \leq \hat \rho_{J,\ell}/\minRho$,
and whenever $N >  C \left(D + u + \log(J)\right)\left(\frac{\maxStandardNorm}{\minRho} \vee \varepsilon^{-2}\right)$ we get
\begin{equation}
\label{eq:mim_lower_bound_local_samples}
\SN{\CX_{J,\ell}} = \hat \rho_{J,\ell}N >
C \left(D + u + \log(J)\right)\left(\Vert \centerRV{Z}|_{J,\ell}\Vert_{\psi_2}^4 \vee \varepsilon^{-2}\hat \rho_{J,\ell}\right).
\end{equation}
Now we can use Lemma \ref{lem:linear_regression_bounds} to concentrate $\Vert \hat b_{J,\ell} - b_{J,\ell}\Vert$ to obtain
\begin{align}
\label{eq:conditioning_for_rho_3}
\bbP\left(\forall \ell \in [J]:\N{\hat b_{J,\ell} - b_{J,\ell}} \leq \left(\varepsilon \sqrt{\frac{\kappa_{J,\ell}}{\hat \rho_{J,\ell}}}\wedge 1 \right) \eta_{J,\ell} \right) \geq 1- \exp(-u).
\end{align}
Finally, \eqref{eq:conditioning_for_rho_3} and $\N{b_{J,\ell}} \leq 2 \eta_{J,\ell}$ from Lemma \ref{lem:linear_regression_bounds} implies
\begin{align*}
T_1 \leq \sum_{\ell=1}^{J} \hat \rho_{J,\ell} \left(\N{\hat b_{J,\ell} - b_{J,\ell}} + 2\N{b_{J,\ell}}\right)\N{\hat b_{J,\ell} -  b_{J,\ell}}
\leq 3\varepsilon\sum_{\ell=1}^{J} \sqrt{\kappa_{J,\ell}}\eta_{J,\ell}^2
\end{align*}
whenever \eqref{eq:conditioning_for_rho_2} and \eqref{eq:conditioning_for_rho_3} hold, \emph{i.e.}
with probability at least $1-3\exp(-u)$ by the union bound. The result follows now from
bounds on $T_1$ and $T_2$, which hold with probability at least $1-4\exp(-u)$, or $1-\exp(-u)$, when adjusting
$C$ in the statement accordingly.
\end{proof}

\begin{theorem}
\label{thm:M_J_perp_bound}
Let \ref{ass_MIM:independence} - \ref{ass_MIM:subGauss} hold.  Fix $u > 0$, $\varepsilon > 0$, and define
$$
\maxStandardNorm := \max_{\ell \in [J]}\Vert \centerRV{Z}|_{J,\ell}\Vert_{\psi_2}^4, \quad \minRho := \min_{\ell \in [J]}\rho_{J,\ell} .
$$
Whenever $N > C(D + u + \log(J))(\maxStandardNorm\ \minRho^{-1} \vee \varepsilon^{-2})$ we have
\begin{align*}
\bbP\left(\N{\left(\hat M_J - M_J\right)Q_J}_F \leq \varepsilon \sum_{\ell=1}^{J}\sqrt{\rho_{J,\ell}\kappa_{J,\ell}}\eta_{J,\ell}\eta_{J,\ell}^{\perp}\right) \geq 1 - \exp(-u).
\end{align*}
\end{theorem}

\begin{proof}
By the definition of $Q_J$ and $Q_{J,\ell}$, we have $\Im(Q_{J})\subseteq \Im(Q_{J,\ell})$. This first allows us to bound
\begin{align*}
 & \N{(\hat M_J - M_J)Q_J}_F = \N{\sum_{\ell=1}^{J} \hat \rho_{J,\ell} \hat b_{J,\ell} \hat b_{J,\ell}^\top Q_J}_F
\leq \sum_{\ell=1}^{J} \hat \rho_{J,\ell} \N{\hat b_{J,\ell}}\N{Q_J\hat b_{J,\ell} }\\
\leq \ & \sum_{\ell=1}^{J} \hat \rho_{J,\ell} \N{\hat b_{J,\ell}}\N{Q_{J,\ell}\hat b_{J,\ell}}
\leq  \sum_{\ell=1}^{J} \hat \rho_{J,\ell} \left(\N{b_{J,\ell}} + \N{\hat b_{J,\ell} - b_{J,\ell}}\right)\N{Q_{J,\ell}\hat b_{J,\ell}}.
\end{align*}
By the same argument as in the proof of Theorem \ref{thm:M_J_parallel_bound}, we have $\SN{\hat \rho_{J,\ell} - \rho_{J,\ell}} \leq 1/2 \rho_{J,\ell}$
for all $\ell \in [J]$ with probability at least $1-2\exp(-u)$, and thus the number
of samples in each level set satisfies \eqref{eq:mim_lower_bound_local_samples}.
Using this together with \eqref{eq:lin_reg_P} and  \eqref{eq:lin_reg_Q},
and the union bound over $\ell \in [J]$, we get
\begin{align*}
&\bbP\left(\forall \ell \in [J] : \N{\hat b_{J,\ell} - b_{J,\ell}} \leq \eta_{J,\ell}\right) \geq 1 - \exp(-u),\\
&\bbP\left(\forall \ell \in [J] : \N{Q_{J,\ell}\hat b_{J,\ell}} \leq \varepsilon  \sqrt{\frac{\kappa_{J,\ell}}{\hat \rho_\ell}}\eta_{J,\ell}^{\perp}\right) \geq 1 - \exp(-u).
\end{align*}
Plugging this, and $\Vert b_{J,\ell} \Vert \leq 2 \eta_\ell$ by Lemma \ref{lem:linear_regression_bounds},
in the initial decomposition, we get with probability at least $1-4\exp(-u)$
\begin{align*}
\N{(\hat M - M)Q_J}_F&\leq  \sum_{\ell=1}^{J} \hat \rho_{J,\ell} \left(\N{b_{J,\ell}} + \N{\hat b_{J,\ell} - b_{J,\ell}}\right)\N{Q_{J,\ell}\hat b_{J,\ell}}\\
&\leq 3\varepsilon \sum_{\ell=1}^{J} \sqrt{\hat \rho_{J,\ell}\kappa_{J,\ell}} \eta_{J,\ell}^{\perp}\eta_{J,\ell}
\leq \frac{9}{2}\varepsilon \sum_{\ell=1}^{J} \sqrt{\rho_{J,\ell}\kappa_{J,\ell}} \eta_{J,\ell}^{\perp}\eta_{J,\ell},
\end{align*}
where the last step follows from $\hat \rho_{J,\ell}\kappa_{J,\ell} \leq 3/2 \rho_{J,\ell}$ for all $\ell \in [J]$.
By suitable choice of $C$ in the statement, we can absorb the factor $9/2$, and adjust the probability to $1-\exp(-u)$.
\end{proof}
\noindent
The guarantee for $\MIMPJ{J}(\tilde d)$ now follows as a Corollary.
\begin{corollary}
\label{cor:mim_projection_error}
Let \ref{ass_MIM:independence} - \ref{ass_MIM:subGauss} hold.
Fix $u > 0$, $\varepsilon > 0$, and define $\maxStandardNorm := \max_{\ell \in [J]}\Vert \centerRV{Z}|_{J,\ell}\Vert_{\psi_2}^4$,
$\minRho := \min_{\ell \in [J]}\rho_{J,\ell} $.
If $\rank(M_J) = \tilde d$ and $\lambda_{\tilde d}(M_J) > \gamma_J > 0$ we have
\begin{align}
\label{eq:error_complete bound}
&\N{\MIMPJ{J}(\tilde d) - P_J}_F  \leq \varepsilon\frac{ \sum_{\ell=1}^{J}\sqrt{\rho_{J,\ell}}\kappa_{J,\ell}\eta_{J,\ell}\eta_{J,\ell}^{\perp}}{\gamma_J},\quad \textrm{whenever}\\
\label{eq:N_requirement_complete_bound}
&N > C(D + u + \log(J))\left(\frac{\maxStandardNorm}{\minRho}
\vee \left(\frac{\sum_{\ell=1}^{J}\sqrt{\kappa_{J,\ell}}\eta_{J,\ell}^2}{\gamma_J}\right)^2 \vee \frac{1}{\varepsilon^2}\right).
\end{align}
\end{corollary}
\begin{proof}
Using Weyl's bound $\lambda_{\tilde d}(\hat M_J)   \geq \lambda_{\tilde d}( M_J ) - \N{\hat M_J - M_J}$ \cite{weyl1912asymptotische}, and Theorem \ref{thm:M_J_parallel_bound}, we have
with probability $1-\exp(-u)$ the guarantee $\lambda_{\tilde d}(\hat M_J) \geq \gamma_J - \N{\hat M_J - M_J}> \frac{1}{2}\gamma_J$,
whenever the number of samples exceeds
\begin{align*}
N > C(D + u + \log(J))\left(\frac{\maxStandardNorm}{\minRho} \vee \left(\frac{\sum_{\ell=1}^{J}\sqrt{\kappa_{J,\ell}}\eta_{J,\ell}^2}{\gamma_J}\right)^2\right)
\end{align*}
Furthermore, Theorem \ref{thm:M_J_perp_bound}
implies
\begin{align*}
\bbP\left(\N{\left(\hat M_J - M_J\right)Q_J}_F \leq \varepsilon \sum_{\ell=1}^{J}\sqrt{\rho_{J,\ell}\kappa_{J,\ell}}\eta_{J,\ell}\eta_{J,\ell}^{\perp}\right) \geq 1-\exp(-u),
\end{align*}
whenever $N > C(D + u + \log(J))(\maxStandardNorm \minRho^{-1}  \vee \varepsilon^{-2})$.
Using the union bound over both events, the conclusion in the statement follows with probability at least $1-2\exp(-u)$
from $\Vert \MIMPJ{J} - P_J\Vert_F \leq \sqrt{2}\Vert\MIMPJ{J}(\tilde d)Q_J\Vert_F$
(see Lemma \ref{lem:projection_equality_frobenius} in the Appendix), and the Davis-Kahan bound \eqref{eq:bound_after_davis_kahan}.
\end{proof}

\noindent
Assuming $\varepsilon^{-2}$ maximizes \eqref{eq:N_requirement_complete_bound},
Corollary \ref{cor:mim_projection_error} implies the bound
\begin{equation}
\label{eq:projection_error_reformulated_bound}
\N{\MIMPJ{J}(\tilde d) - P_J}_F \leq C \frac{\sqrt{1+\log(J)}\sum_{\ell=1}^{J}\sqrt{\rho_{J,\ell}\kappa_{J,\ell}}\eta_{J,\ell}\eta_{J,\ell}^{\perp}}{\gamma_J} \sqrt{\frac{D+u}{N}}.
\end{equation}
It separates the error into a leading factor, which only depends on the
hyperparameter $J$, respectively, the induced level set partition,
and a trailing factor, which describes
dependencies on $\sqrt{D}$, $N^{-1/2}$ and the confidence parameter $u$.
By using anisotropic bounds from Lemma \ref{lem:linear_regression_bounds},
we obtain a linear dependence on $\eta_{J,\ell}$, which scales like the term
$\Vert \hat b_{J,\ell}\Vert$, and a linear dependence
on $\eta_{J,\ell}^{\perp}$, which scales like $\Vert Q_{J}\hat b_{J,\ell}\Vert $. An isotropic
concentration bound for $\hat b_{J,\ell} - b_{J,\ell}$ would have instead given $\eta_{J,\ell}^2$,
which can be significantly worse judging by observations in Figure \ref{fig:levelset_details}.

\subsubsection{Data-driven proxy and tightness of \eqref{eq:projection_error_reformulated_bound}}
\label{subsubsec:mim_data_driven_proxy_and_tightness}
\begin{figure}[h]
        \centering
        \begin{subfigure}[t]{0.32\textwidth}
                \centering
                \includegraphics[width=\textwidth]{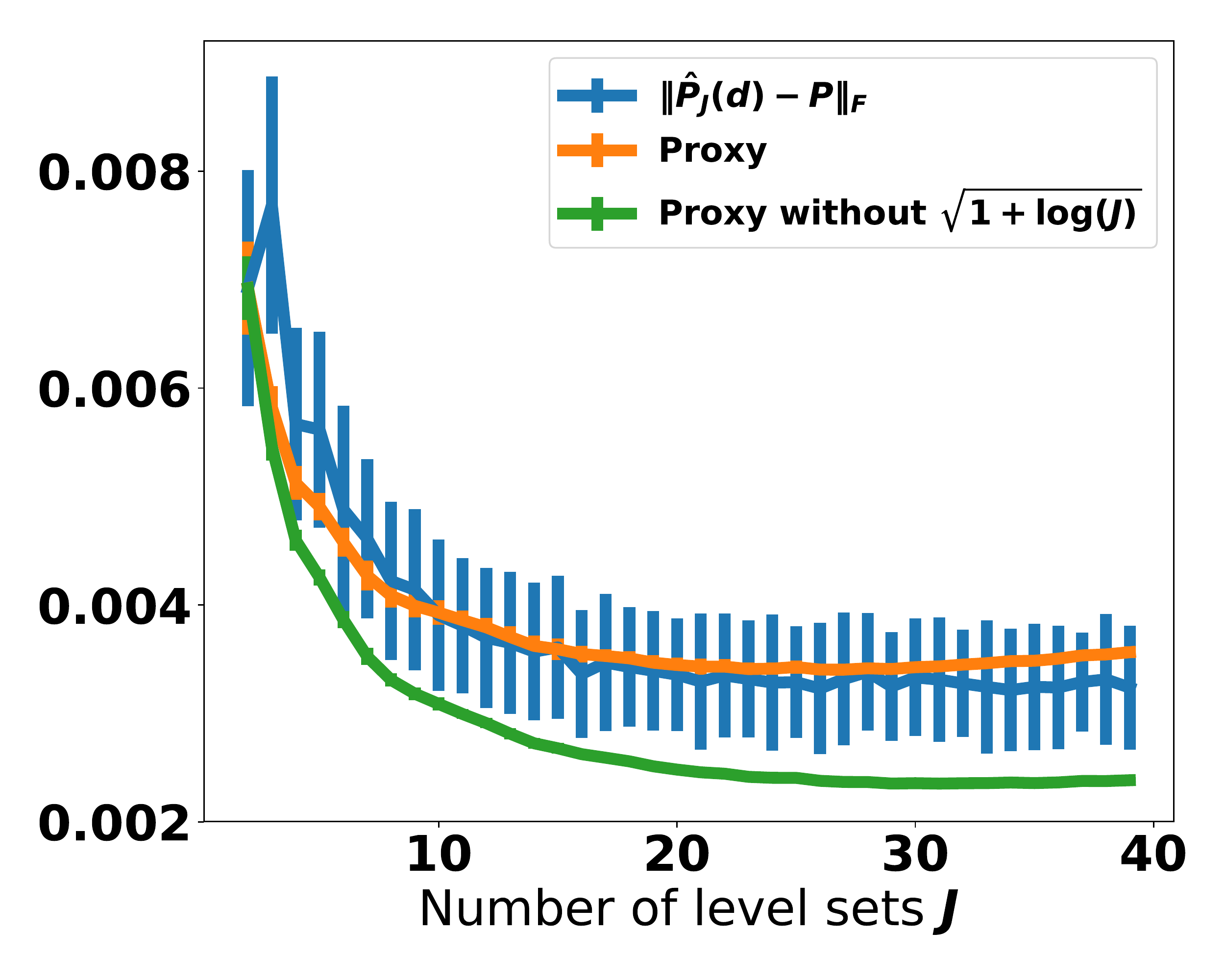}
                \caption{\scriptsize $g(x) = \frac{x_2}{1+(x_1 - 1)^2}$}
                \label{fig:mim_sim}
        \end{subfigure}
        \begin{subfigure}[t]{0.32\textwidth}
                \centering
                \includegraphics[width=\textwidth]{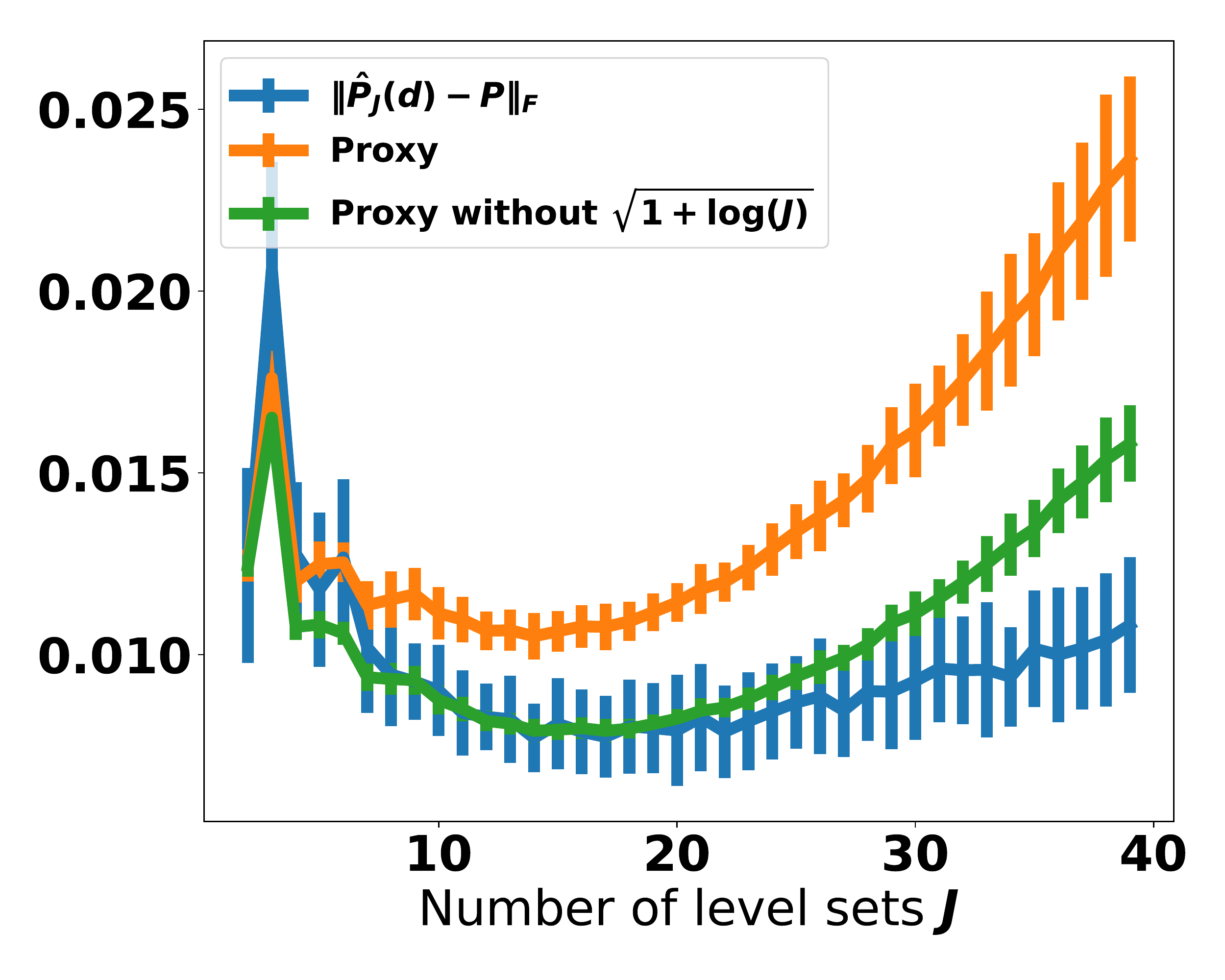}
                \caption{\scriptsize $g(x) = 0 \vee (x_1 - \frac{1}{10}) + \frac{x_2+1}{2}$}
                \label{fig:relu_net}
        \end{subfigure}
        \begin{subfigure}[t]{0.32\textwidth}
                \centering
                \includegraphics[width=\textwidth]{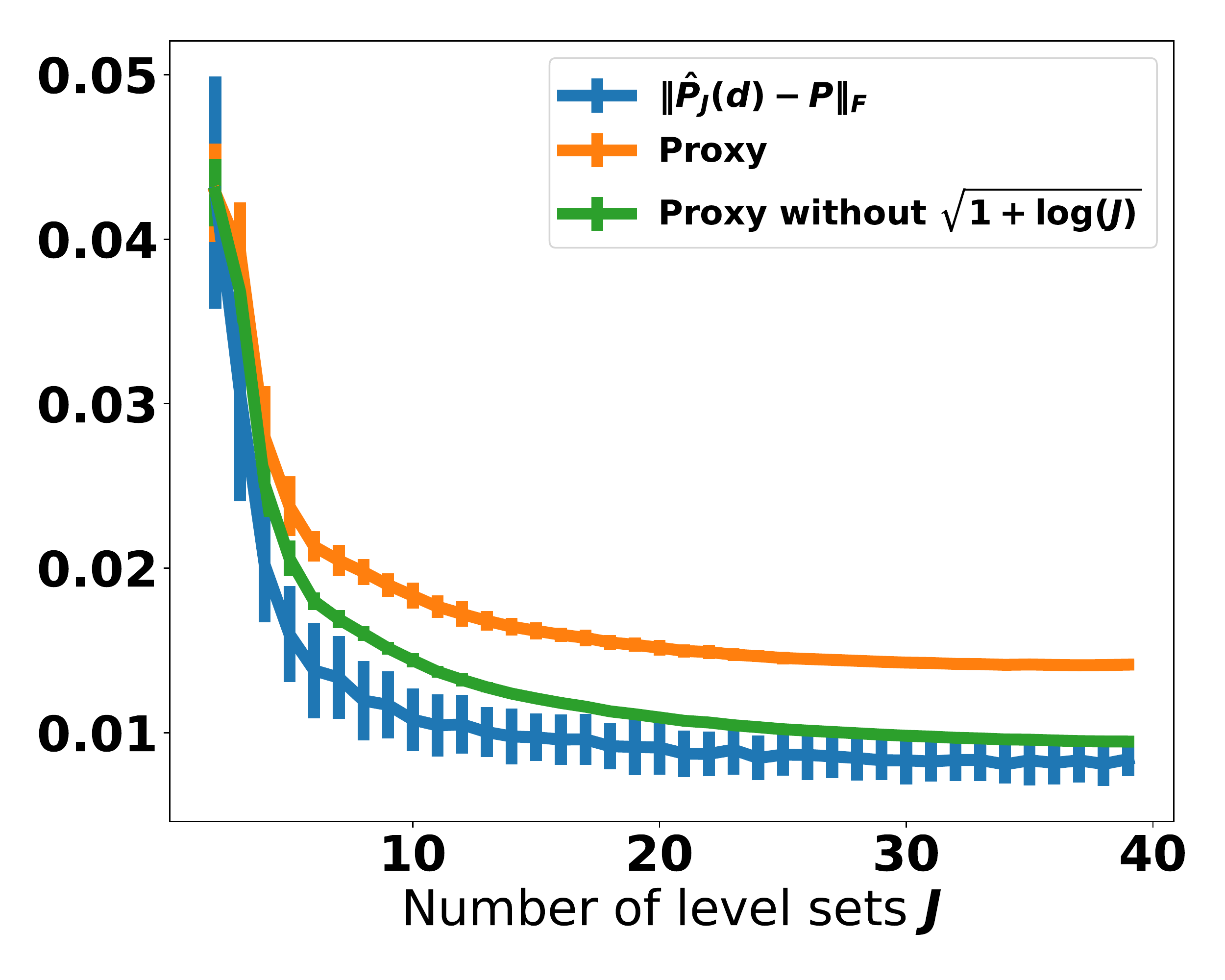}
                \caption{\scriptsize $g(x) = x_2e^{\sin(x_1)x_2 + x_2}$}
                \label{fig:mim_exp_sinus}
        \end{subfigure}
        \begin{subfigure}[t]{0.32\textwidth}
                \centering
                \includegraphics[width=\textwidth]{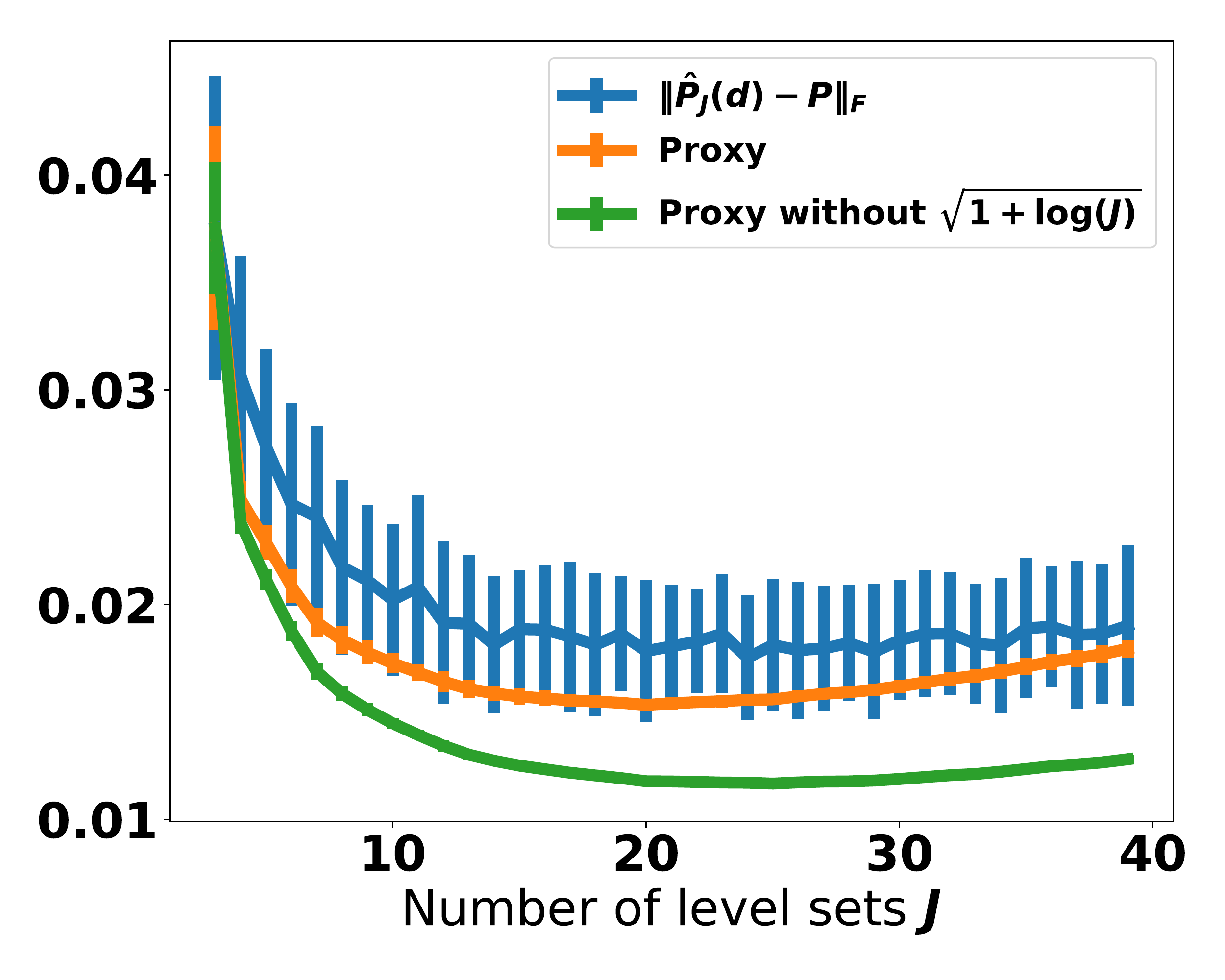}
                \caption{\scriptsize $g(x) = x_2e^{\sin(x_1)x_2 + x_3}$}
                \label{fig:mim_exp_sinus_3}
        \end{subfigure}
        \begin{subfigure}[t]{0.32\textwidth}
                \centering
                \includegraphics[width=\textwidth]{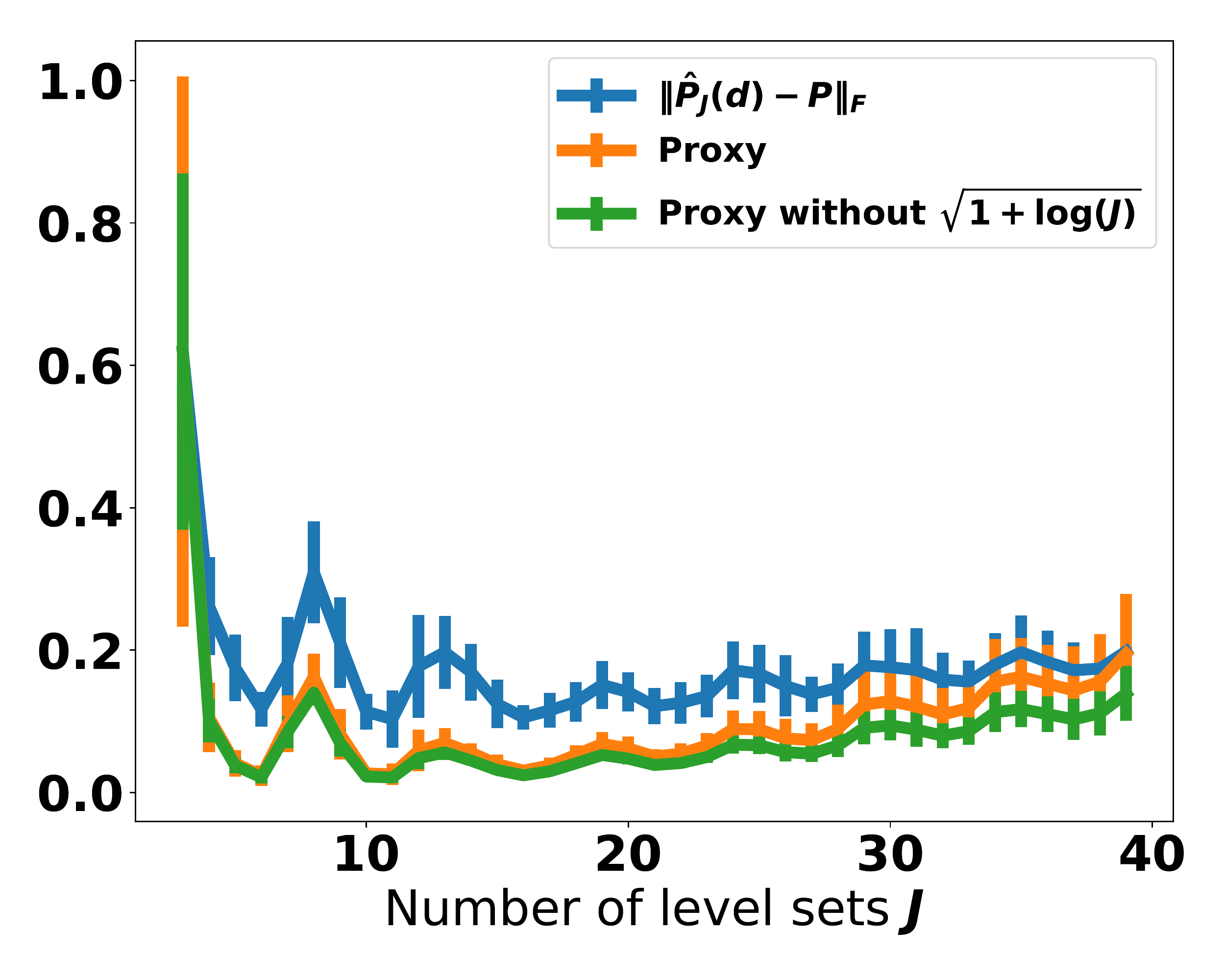}
                \caption{\scriptsize Function \eqref{eq:mim_example_pwlin}}
                \label{fig:mim_pw_lin}
        \end{subfigure}
        \begin{subfigure}[t]{0.32\textwidth}
                \centering
                \includegraphics[width=\textwidth]{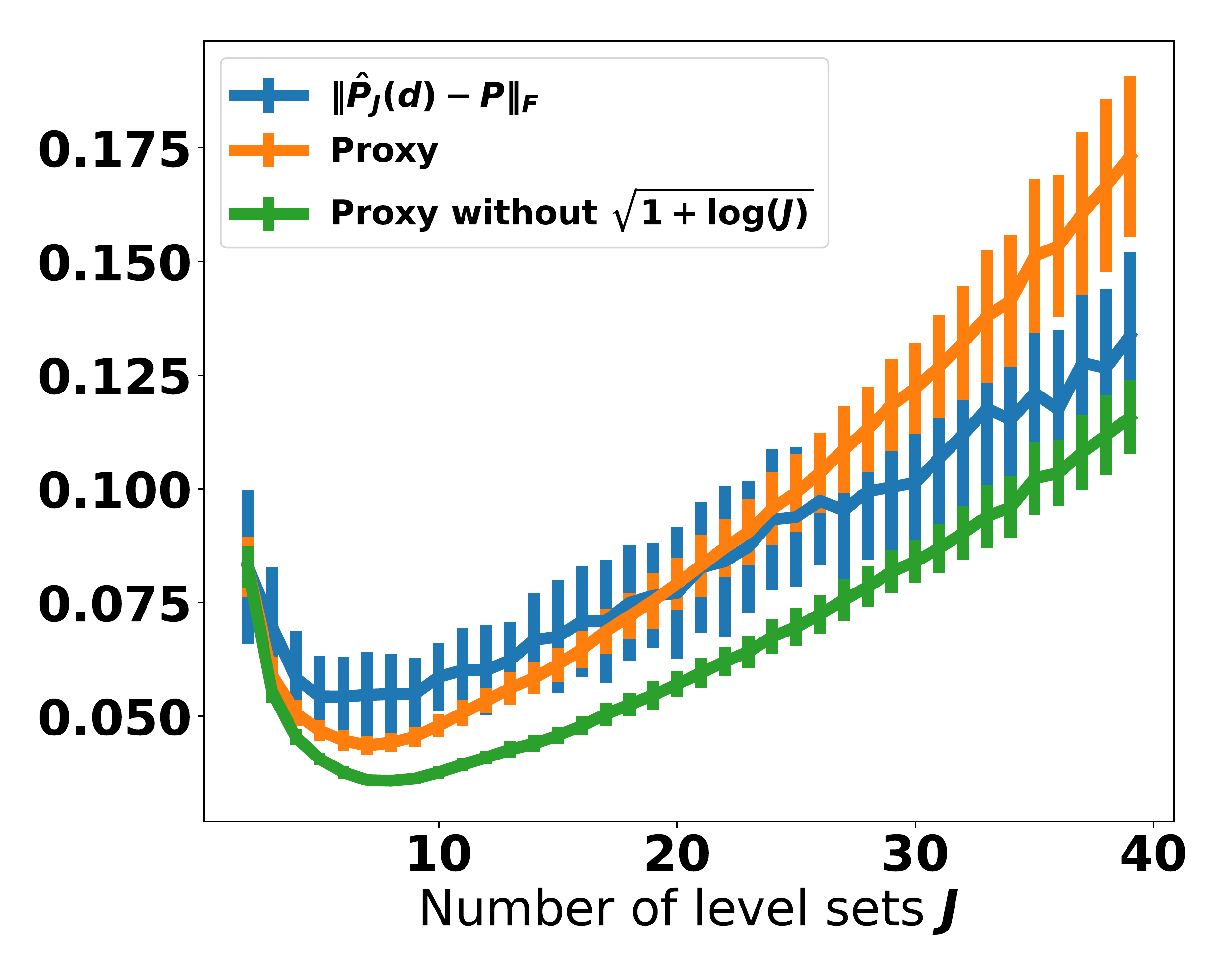}
                \caption{\scriptsize $g(x) = \frac{\sin(x_1) + \cos(x_2 - \frac{1}{4})}{(1+x_1^2)}$}
                \label{fig:mim_div_sincos}
        \end{subfigure}
        \caption{Figures plot the function $J\mapsto \Vert \MIMPJ{J}(d) - \MIMPTrue\Vert_F$ (blue),
        and compare it to the right hand side of \eqref{eq:mim_data_driven_proxy} with (orange) and without (green) the union bound factor
        $\sqrt{1+\log(J)}$.
        The proxy is rescaled by a constant to match $\Vert \MIMPJ{J}(d) - \MIMPTrue\Vert_F$ for $J=d$. Vertical bars indicate the standard deviation.}
\end{figure}

We now empirically study the tightness of \eqref{eq:projection_error_reformulated_bound}
when considering a fixed number of samples $N$ but varying the number of level sets $J$.
First, we develop a data-driven proxy to estimate the leading factor in \eqref{eq:projection_error_reformulated_bound}
from a given data set. Afterwards, we compare the proxy with the true error
on several synthetic examples.

\paragraph*{Data-driven proxy}
We have to replace $\gamma_J$, $\rho_{J,\ell}$, $\eta_{J,\ell}$, $\eta_{J,\ell}^{\perp}$
and $\kappa_{J,\ell}$ in \eqref{eq:projection_error_reformulated_bound} by quantities
that can be estimated from data. The first three quantities are approximated by
$\gamma_J \approx \lambda_{d}(\hat M_J)$, $\rho_{J,\ell} \approx \hat \rho_{J,\ell}$
and $\eta_{J,\ell} \approx \Vert \hat b_{J,\ell}\Vert$, where the last replacement
is motivated by the fact that $\eta_{J,\ell}$ is used to bound $\Vert \hat b_{J,\ell}\Vert \leq \Vert b_{J,\ell}\Vert + \Vert b_{J,\ell}-\hat b_{J,\ell}\Vert \lesssim \eta_{J,\ell}$
in the proofs of Theorems \ref{thm:M_J_parallel_bound} and \ref{thm:M_J_perp_bound}.
Furthermore, we use the conditional sample covariance $\hat \Sigma_{J,\ell}$, and
projections $\hat P_{J,\ell}:=\textN{\hat b_{J,\ell}}^{-2}\hat b_{J,\ell}\hat b_{J,\ell}^\top$ and  $\hat Q_{J,\ell} := \Id - \hat P_{J,\ell}$,
to compute an approximation to $\kappa_{J,\ell}$ by
\begin{align*}
\hat \kappa_{J,\ell}:=\max\left\{\N{\hat P_{J,\ell} \hat \Sigma_{J,\ell} \hat P_{J,\ell}}\N{\hat P_{J,\ell} \hat \Sigma_{J,\ell}^{\dagger} \hat P_{J,\ell}},
\N{\hat Q_{J,\ell} \hat \Sigma_{J,\ell} \hat Q_{J,\ell}}\N{\hat Q_{J,\ell} \hat \Sigma_{J,\ell}^{\dagger} \hat Q_{J,\ell}}\right\}.
\end{align*}
Note that replacing squared sub-Gaussian norms with spectral norms of the corresponding covariance
matrices can underestimate the true value of $\kappa_{J,\ell}$.
The same strategy is used for $\eta_{J,\ell}^{\perp}$, \emph{i.e.} we approximate $\eta_{J,\ell}^{\perp}$ by
\begin{align*}
\hat \eta_{J,\ell}^{\perp} := \sqrt{\lsmean{\CY_{J,\ell}}{\left(Y - \lsmean{\CY_{J,\ell}}{Y}\right)^2}\N{\hat Q_{J,\ell}\hat \Sigma_{J,\ell}^{\dagger}\hat Q_{J,\ell}}}.
\end{align*}
Combining everything, the data-driven proxy for the leading factor in \eqref{eq:projection_error_reformulated_bound}
without the union bound factor $\sqrt{1+\log(J)}$ is given by
\begin{equation}
\label{eq:mim_data_driven_proxy}
\frac{\sum_{\ell=1}^{J}\sqrt{\rho_{J,\ell}\kappa_{J,\ell}}\eta_{J,\ell}\eta_{J,\ell}^{\perp}}{\gamma_J}  \approx
\frac{\sum_{\ell=1}^{J}\sqrt{\hat \rho_{J,\ell}\hat \kappa_{J,\ell}}\N{\hat b_{J,\ell}}\hat \eta_{J,\ell}^{\perp}}{\lambda_{d}(\hat M_J)}.
\end{equation}
In order to reduce the variance in estimating \eqref{eq:mim_data_driven_proxy},
we further restrict the sum to level sets with at least $\SN{\CX_{J,\ell}} > 5D$ samples in the experiments below.

\paragraph*{Experiments} We sample $N = 80000$ points from $\Uni{\{X : \N{X}\leq 1\}}$ in $D = 20$ dimensions
and set $Y = g(A^\top X) + \zeta$, where $\zeta \sim \CN(0,10^{-4}\textrm{Var}(f(X)))$ and $A = [e_1|\ldots|e_d]$ with $e_i$ being the
$i$-th standard basis vector.
Each experiment is repeated 50 times and we report averaged results plus standard
deviations for different  link functions in Figures \ref{fig:mim_sim}-\ref{fig:mim_div_sincos}.

We observe that the map $J \mapsto \Vert \MIMPJ{J}(d) - P\Vert_F$ initially decreases
when increasing the number of level sets $J$ beyond $d$, and then either stalls, such
as in Figures \ref{fig:mim_sim}, \ref{fig:mim_exp_sinus}, \ref{fig:mim_pw_lin} and \ref{fig:mim_exp_sinus_3}, or increases as in
Figures  \ref{fig:relu_net} and \ref{fig:mim_div_sincos}. This behavior is captured well
by the data driven proxy \eqref{eq:mim_data_driven_proxy}. Furthermore,
even if the relation $J \mapsto \Vert \MIMPJ{J}(d) - P\Vert_F$ shows kinks as in Figure \ref{fig:mim_pw_lin},
where the link function is given by $g(x_1,x_2,x_3) = \sum_{i=1}^{3}g_i(x_i)$ with
\begin{equation}
\label{eq:mim_example_pwlin}
\begin{aligned}
g_1(x_1) &= 1(x_1 < 0) 0.1 x_1 + 1(x_1 \geq 0) 2.0\left(x_1 - 0.5\right) + 0.05,\\
g_2(x_2) &= 1(x_2 < 0) 2.0x_2 + 1(x_2 \geq 0) 0.1\left(x_2 - 0.5\right),\\
g_3(x_3) &= 1(x_3 < 0) 5.0x_3 + 1(x_3 \geq 0) 0.1 \left(x_3 - 0.2\right) + 1,
\end{aligned}
\end{equation}
the derived data-driven proxy reproduces the same behavior.
The experiments suggest that Corollary \ref{cor:mim_projection_error} characterizes the influence
of $J$ and the induced level set partition on the projection error well.
Furthermore, they raise the question whether $J$, which minimizes the data-driven proxy \eqref{eq:mim_data_driven_proxy}, can be used
for hyperparameter tuning in practice.
This is an interesting direction for future work, because choosing $J$ for the related class of inverse regression based methods has been identified as a notoriously
difficult problem, for which no good strategies exist \cite{ma2013review}.

\section{Regression in the reduced space}
\label{sec:mim_regression}
In this section we return to the multi-index model $Y = g(A^\top X) + \zeta$ with $\bbE[\zeta|X] = 0$
almost surely. Assumption $\zeta \independent X|A^\top X$ is not strictly required in this part. The second step to estimate the model is to learn the link function $g$, while leveraging the approximated projection $\MIMP \approx \MIMPTrue$, \emph{e.g.} constructed by using RCLS.
We restrict our analysis to two popular and commonly used regressors, namely
kNN-regression and piecewise polynomial regression. Our analysis reveals how the
error $\textN{\MIMP-\MIMPTrue}$ affects kNN and piecewise polynomials, if they are trained
on perturbed data $\{(\MIMP X_i, Y_i) :  i \in [N]\}$ instead of $\{(\MIMPTrue X_i, Y_i): i \in [N]\}$.
For simplicity,
we assume $\MIMP$ is deterministic and thus statistically independent of $\{(X_i,Y_i) : i \in [N]\}$.
In practice, statistical independency can be ensured by using separate data sets for learning
$\MIMP$ and performing the subsequent regression task.

To study regression rates, smoothness properties of the link function play an important role. We use
to the following standard definition \cite{gyorfi2006distribution}.
\begin{definition}
\label{def:smoothness}
Let $f : \bbR^{D} \rightarrow \bbR$, $s_1 \in \bbN_{0}$, $s_2 \in (0,1]$ and $s=s_1+s_2$. We say $f$ is $(L,s)$-smooth if
partial derivatives $\partial^{\alpha}f$ exist for all $\alpha \in \bbN_0^{D}$ with $\sum_{i}\alpha_i\leq s_1$,
and for all $s$ with $\sum_{i}\alpha_i  = s_1$ we have
\begin{align*}
\SN{\partial^{\alpha}f(z) - \partial^{\alpha}(z')} \leq L\N{z - z'}^{s_2}.
\end{align*}
\end{definition}

\noindent
The minimax rate for nonparametric estimation in $\bbR^{d}$ is well known \cite{gyorfi2006distribution,S98} and reads,
for $(L,s)$-smooth regression function $f$,
\begin{align}
\label{eq:minimax_rate_MIM}
\MSE{\hat f}{f} := \bbE \left(\hat f (X) - f(X)\right)^2 \asymp N^{-\frac{2s}{2s+d}}.
\end{align}
Similarly, the rate is a lower bound for nonparametric estimation of the multi-index model with $\dim(\MIMPTrue) =d$, because
we are still left with a nonparametric regression problem in $\bbR^d$ once $\MIMPTrue$ is identified.
In the following, we provide conditions on $\textN{\MIMP-\MIMPTrue}$ so that
the optimal rate \eqref{eq:minimax_rate_MIM} is achieved, when training on perturbed data. In the analysis,
we assume that $X$ is sub-Gaussian,
$\SN{f(X)}\leq 1$ almost surely, and $\textrm{Var}(\zeta|X) \leq \sigma_{\zeta}^2$ almost surely.

\subsection{kNN-regression}
\label{subsec:MIM_kNN_regression}
Let $x$ be a new data point and denote a reordering of the indices by $1(x),\ldots,N(x)$ so that
\begin{align*}
\N{\MIMP \left(x - X_{i(x)}\right)} \leq \N{\MIMP\left(x - X_{j(x)}\right)} \textrm{ for all } j \geq i \textrm{ and all } i,
\end{align*}
\emph{i.e.} $i(x)$ is the $i$-th nearest neighbor to $x$ after projecting onto $\Im(\MIMP)$.
The kNN-estimator is defined by $\hat f_k(x) := k^{-1}\sum_{i=1}^{k} Y_{i(x)}$ and
the following theorem characterizes the influence of the projection error on the generalization performance.
The proof resembles \cite{gyorfi2006distribution, kohler2006rates} and is given
in Appendix \ref{subsec:app_kNN_mim_proof}.
\begin{theorem}
\label{thm:regression_error_kNN}
Let $g$ be $(L,s)$-smooth for $s \in (0,1]$, and $d > 2s$. For $k = C_k N^{2s/(2s+d)}$,
we obtain
\begin{equation}
\label{eq:guarantee_kNN}
\MSE{\hat f_k}{f} \leq C_1 N^{-\frac{2s}{2s+d}} + C_2\log(N) \N{\MIMP-\MIMPTrue}^{2s},
\end{equation}
where $C_1$ depends on $d,\sigma_{\zeta},\N{X}_{\psi_2},C_k,L,s$, and $C_2$ additionally linearly on $D\N{X}_{\psi_2}^2$.
\end{theorem}
\begin{remark}[$d > 2s$ assumption]
\label{rem:d_2s}
The condition $d > 2s$ in Theorem \ref{thm:regression_error_kNN} is not due to
the error $\Vert \MIMP - \MIMPTrue \Vert$, but arises from \cite[Lemma 1]{kohler2006rates},
where ordinary kNN is analyzed for unbounded marginal distributions. It has been shown in
\cite{gyorfi2006distribution} that achieving similar rates for $d\leq 2s$
requires an extra assumption of the marginal distribution of $X$ (boundedness does not suffice).
\end{remark}
\begin{remark}[Rate optimality]
Assuming $\Vert \MIMP - \MIMPTrue\Vert \in \CO(N^{-1/2})$, we observe that
the second term in \eqref{eq:guarantee_kNN} has order $N^{-s}$. Therefore,
Theorem \ref{thm:regression_error_kNN} ensures, up to the logarithmic factor, the optimal rate
$N^{-2s/(2s+d)}$ for $d \geq 2$.
The logarithmic factor disappears, if the marginal distribution of $X$ is bounded.
\end{remark}

\subsection{Piecewise polynomial regression}
\label{subsec:piecewise_polynomials}
Piecewise polynomial estimators can be defined in different ways as they depend on a partition of the underlying space.
Therefore we first have to describe the type of piecewise polynomials that we consider in the following.

Let $\hat A \in \bbR^{D\times d}$ contain column-wise an arbitrary orthonormal basis of $\Im(\MIMP)$.
Denote by $\Delta_l$ the set of dyadic cubes in $\bbR^{d}$, \emph{i.e.} the set of cubes with side length $2^{-l}$ and
corners in the set $\{2^{-l}(v_1,\ldots,v_d) : v_j \in \bbZ\}$, and let $\Delta_l(R) \subseteq \Delta_l$
be the subset that has non-empty intersection with $\{\hat A^\top z : z \in B_R\}$,
where $B_R = \{X \in \bbR^D : \N{X} \leq R\}$.
Moreover, let $\CP_{k}$ be the space of polynomials of degree $k$ in $\bbR^{d}$ and $1_{A}$ be
the characteristic function of a set $A$. The function space of piecewise polynomials
we consider is defined by
\begin{align}
\label{eq:function_space}
\CF(\hat A, l, k, R)\!:=\! \left\{f\!:\! f(x)\! =\! 1_{B_{R}}(x)\!\!\!\!\sum\limits_{c \in \Delta_l(R)}\!\!\!1_{c}\!\left(\hat A^\top x\right)\!p_{c}\left(\hat A^\top x\right),p_{c} \in \CP_{k} \right\}\!.
\end{align}
To construct the estimator, we perform empirical risk minimization
\begin{align}
\label{eq:empirical_risk_minimizer}
\tilde f := \argmin_{h\in \CF(\hat A, l,k,R)}\sum\limits_{i=1}^{N}\left(h(X_i) - Y_i\right)^2,
\end{align}
and then set $\hat f(x) := T_{[-1,1]}(\tilde f(x))$, where $T_{[-1,1]}(u) := \textrm{sign}(u)(\SN{u}\wedge 1)$.
Note that piecewise polynomial estimators are typically analyzed after thresholding to avoid
technical difficulties with potentially unbounded predictions (see also \cite{binev2007universal,gyorfi2006distribution}).

The following theorem characterizes the influence of $\Vert \MIMP - \MIMPTrue\Vert$ on the generalization
performance of the estimator.
\begin{theorem}
\label{thm:regression_error_piecewise_pols}
Let $g$ be $(L, s)$-smooth with $s = s_1 + s_2$, $s_1 \in \bbN_0$, $s_2 \in (0,1]$.
Choosing $l = \left\lceil \log_2(N)/(2s+d)\right\rceil$, $R^2 = D\N{X}_{\psi_2}^2 \log(N)$,
and $k = s_1$ we get
\begin{equation}
\begin{aligned}
\label{eq:result_piecewise_polynomials}
\MSE{\hat f}{f} &\!\leq\!C_1\log^{1 \vee \frac{d}{2}}(N) N^{-\frac{2s}{2s+d}} +  C_2 \log(N)^{1\wedge s} \N{\MIMP- \MIMPTrue}^{2\wedge 2s}\!\!,
\end{aligned}
\end{equation}
where the constants grow with $\sigma_{\zeta},d,s$, $L^* := Ld^{s_1/2}(1-\Vert \MIMP - \MIMPTrue\Vert^2)^{-s/2}$,
and $C_1$ depends linearly on $(D\N{X}_{\psi_2}^2)^{d/2}$, and $C_2$ linearly on $(D\N{X}_{\psi_2}^2)^{1 \wedge s}$.
\end{theorem}
\begin{remark}[Boundedness and $\log$-factors]
For bounded $X$, the choice $R^2 \asymp \log(N)$ is not required and
$\log^{1 \vee \frac{d}{2}}(N)$ reduces to $\log(N)$. Moreover, $D\N{X}_{\psi_2}^2$ can
be replaced by the squared radius of a ball containing the support of $X$, which removes the dependency
on $D$ entirely.
\end{remark}
\begin{remark}[Rate optimality]
Assuming $\Vert \MIMP - \MIMPTrue \Vert \in \CO(N^{-1/2})$, we observe that
the last term in \eqref{eq:result_piecewise_polynomials} has order $N^{-s\wedge 1}$. Therefore,
Theorem \ref{thm:regression_error_piecewise_pols} ensures, up to $\log$-factors, the
optimal rate $N^{-2s/(2s+d)}$ for $ d = 1 $ and $ s \geq \frac {1}{2} $, or $ d \geq 2 $ and $ s > 0 $.
\end{remark}

\paragraph*{Proof sketch} The first step is to apply the following well-known result.

\begin{theorem}[Theorem 11.3 in \cite{gyorfi2006distribution}]
\label{lem:combined_bound}
Let $\CF$ be a vector space of functions $f : \bbR^{D} \rightarrow [-1,1]$. Assume
$Y = f(X) + \zeta$, $\bbE[Y|X] = f(X)$ and $\textrm{Var}(\zeta|X=x) \leq \sigma_{\zeta}^2$.
Denote by $\tilde f$ the empirical risk minimizer in $\CF$ over $N$ iid. copies of $(X,Y)$, and let $\hat f = T_{[-1,1]}(\tilde f)$.
Then there exists a universal constant $C$ such that
\begin{align}
\label{eq:combined_bound}
\MSE{\hat f}{f}\leq C(\sigma_{\zeta}^2 \vee 1)\frac{\log(N) + \dim(\CF)}{N} + C\inf\limits_{h \in \CF}\MSE{h}{f}.
\end{align}
\end{theorem}

\noindent
The first term in \eqref{eq:combined_bound} is the estimation error, which measures
the deviation of the performance of the empirical risk minimizer to the best performing estimator in $\CF$
when having access to the entire distribution. It decreases as more samples become available, but increases with the complexity
of $\CF$, here measured in terms of the dimensionality. It can be checked that $\CF(\hat A, l,k,R)$ is closed under addition and scalar multiplication and
is thus a vector space. A basis can be constructed by combining the standard polynomial basis
for each cell of the partition. Therefore $\dim(\CF(\hat A, l,k,R)) = \SN{\Delta_l(R)}{{d+k}\choose k}$,
where $\SN{\Delta_l(R)}$ is the number of cells required to cover $\{\hat A^\top z : z \in B_R\}$. Lemma \ref{lem:bound_number_of_cells} in the
Appendix proves $\SN{\Delta_l(R)} \leq \lceil (2^{l+1}R)^d \rceil$ and therefore
\begin{equation} \label{eq:dimF}
\dim(\CF(\hat A, l,k,R)) \leq {{d+k}\choose k}\lceil(2^{l+1}R)^d\rceil .
\end{equation}

The second term in \eqref{eq:combined_bound} is the approximation error, which measures how
well $f$ can be approximated by any function $h \in \CF(\hat A, l,k,R)$. Neglecting for a moment
the perturbation $\MIMP - \MIMPTrue$, it is known that a piecewise Taylor expansion
of $g$ can be used to approximate $g$ with an accuracy that increases
as the underlying partition is refined. The main difficulty in our case is to define
a piecewise polynomial function $h \in \CF(\hat A, l, k, R)$ that
approximates $g(A^\top x)$, despite the fact that $h$ depends on coordinates $\hat A^\top x$
instead of $A^\top x$.

To define such a function, we first prove the existence of a function
$g^*$ that approximates $g$ uniformly well, when being evaluated on $\hat A^\top x$. Precisely,
Lemma \ref{lem:close_function} in the Appendix shows
\begin{align*}
\SN{g^*(\hat A^\top x) - g(A^\top x)} \leq L^* \N{x}^{1 \wedge s}\Vert \MIMP - \MIMPTrue\Vert^{1 \wedge s},
\end{align*}
for some $(L^*,s)$-smooth function $g^*$. Now, by approximating $g^*$ through a piecewise
Taylor expansion, we can construct a function $h \in \CF(\hat A, l, k, R)$
which, using choices $l,\, k$ and $R$ as in Theorem \ref{thm:regression_error_piecewise_pols}, satisfies
\begin{align}
\label{eq:approximation_error_mean}
\MSE{h}{f} \leq C_1 N^{-\frac{2s}{2s+d}} + C_2 \log^{1 \wedge s}(N) \N{\MIMP - \MIMPTrue}^{2 \wedge 2s},
\end{align}
for constants $C_1$ depending on $L^*, d, s$, and $C_2$ depending on $L^*$ and linearly on
$(D\N{X}_{\psi_2}^2)^{1\wedge s}$ (see Corollary \ref{cor:mean_approximation_error}).
The proof of Theorem \ref{thm:regression_error_piecewise_pols}
concludes by combining Theorem \ref{lem:combined_bound}, the
dimensionality bound \eqref{eq:dimF},
and the approximation error bound \eqref{eq:approximation_error_mean} (see Appendix \ref{subsec:mim_pp_proof}).

\section{Numerical experiments}
\label{sec:mim_numerical_examples}

\begin{figure}[h]
        \centering
        \begin{subfigure}[t]{0.32\textwidth}
                \centering
                \includegraphics[width=\textwidth]{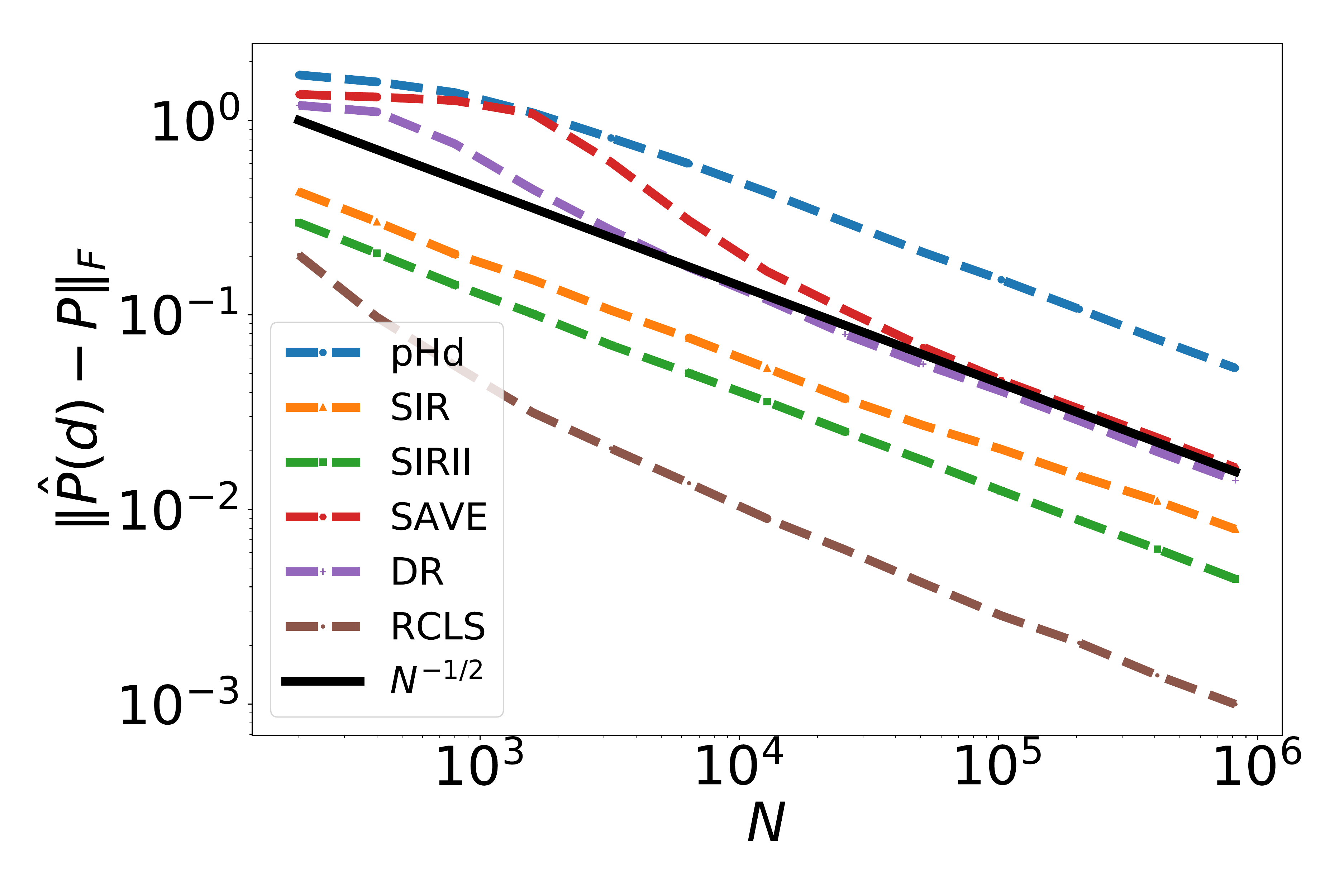}
                \caption{\scriptsize $g(x) = \frac{x_2}{1+(x_1 - 1)^2}$}
                \label{fig:mim_simple_division_comparison}
        \end{subfigure}
        \begin{subfigure}[t]{0.32\textwidth}
                \centering
                \includegraphics[width=\textwidth]{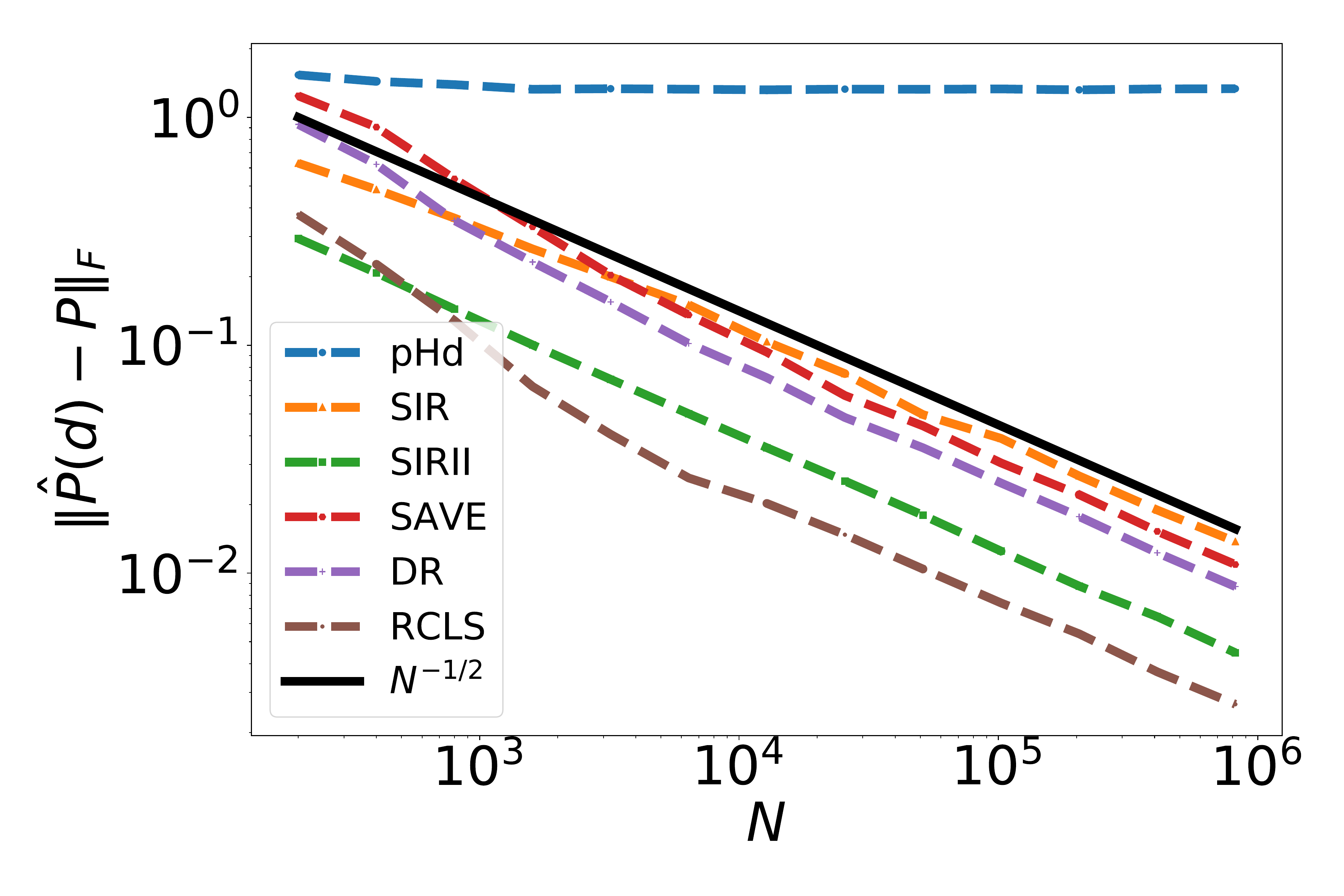}
                \caption{\scriptsize $g(x) = 0 \vee (x_1 - \frac{1}{10}) + \frac{x_2+1}{2}$}
                \label{fig:mim_relu_net_comparison}
        \end{subfigure}
        \begin{subfigure}[t]{0.32\textwidth}
                \centering
                \includegraphics[width=\textwidth]{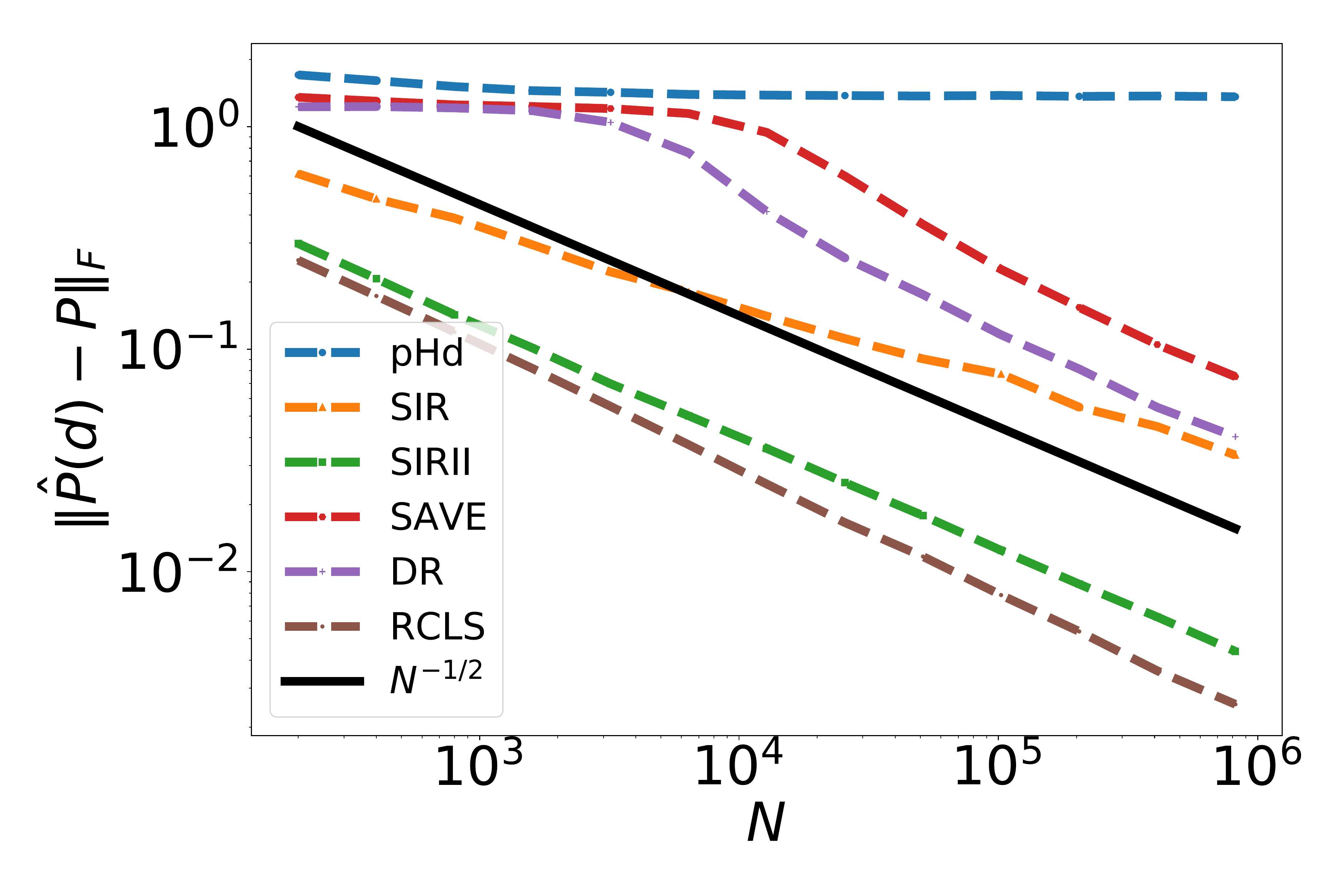}
                \caption{\scriptsize $g(x) = x_2e^{\sin(x_1)x_2 + x_2}$}
                \label{fig:mim_exp_sinus_comparison}
        \end{subfigure}
        \begin{subfigure}[t]{0.32\textwidth}
                \centering
                \includegraphics[width=\textwidth]{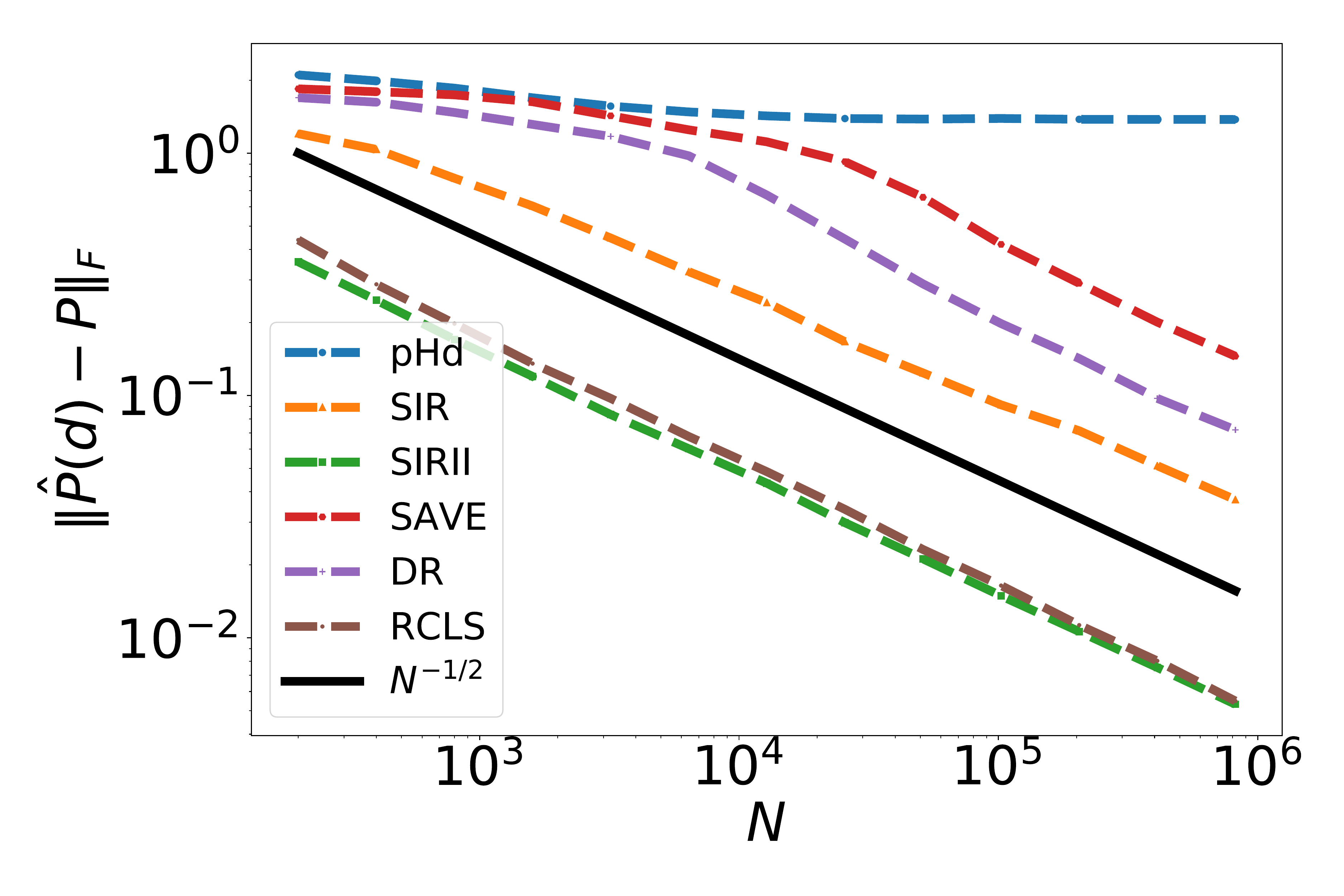}
                \caption{\scriptsize $g(x) = x_2e^{\sin(x_1)x_2 + x_3}$}
                \label{fig:mim_exp_sinus_3_comparison}
        \end{subfigure}
        \begin{subfigure}[t]{0.32\textwidth}
                \centering
                \includegraphics[width=\textwidth]{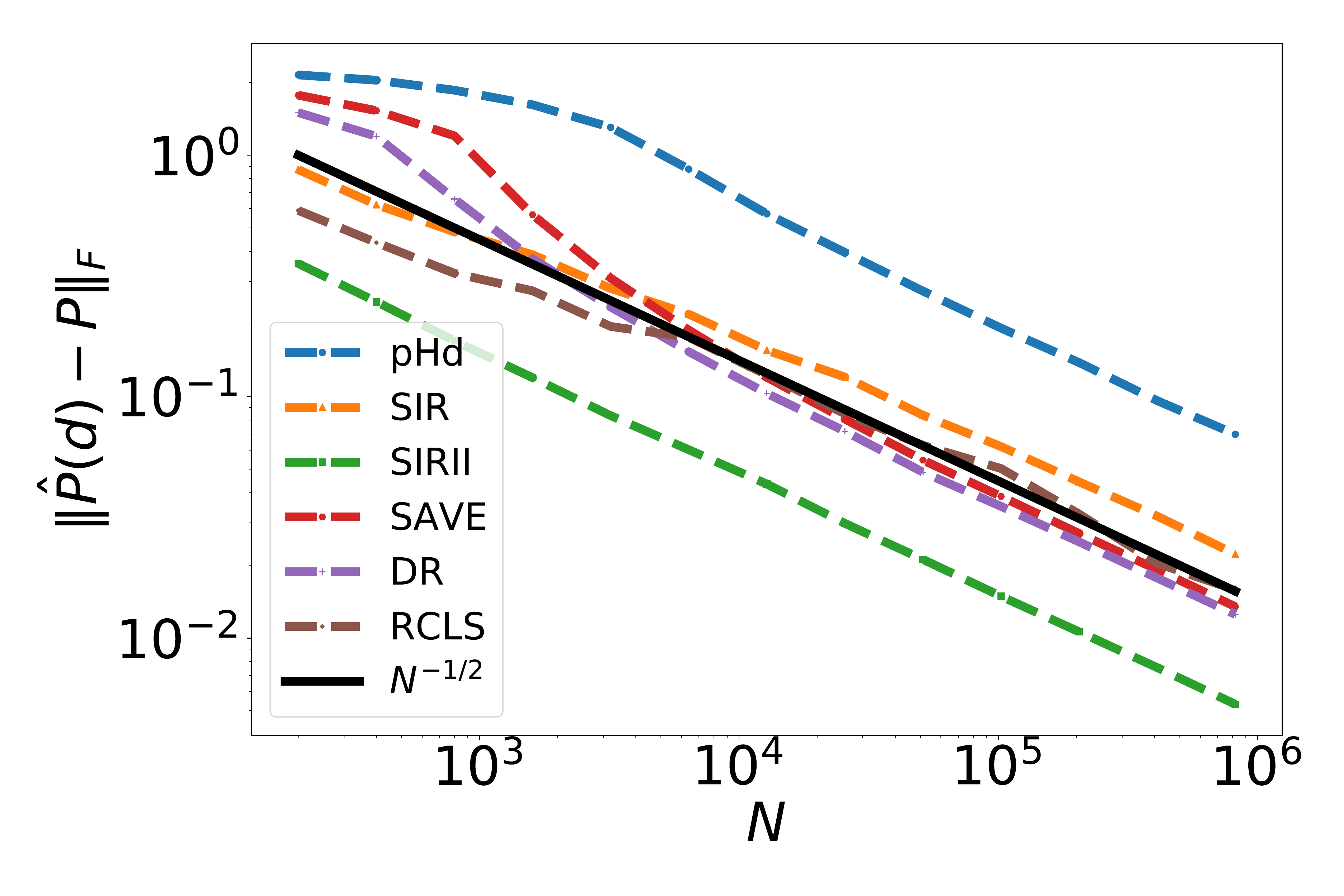}
                \caption{\scriptsize Function \eqref{eq:mim_example_pwlin}}
                \label{fig:mim_pw_lin_comparison}
        \end{subfigure}
        \begin{subfigure}[t]{0.32\textwidth}
                \centering
                \includegraphics[width=\textwidth]{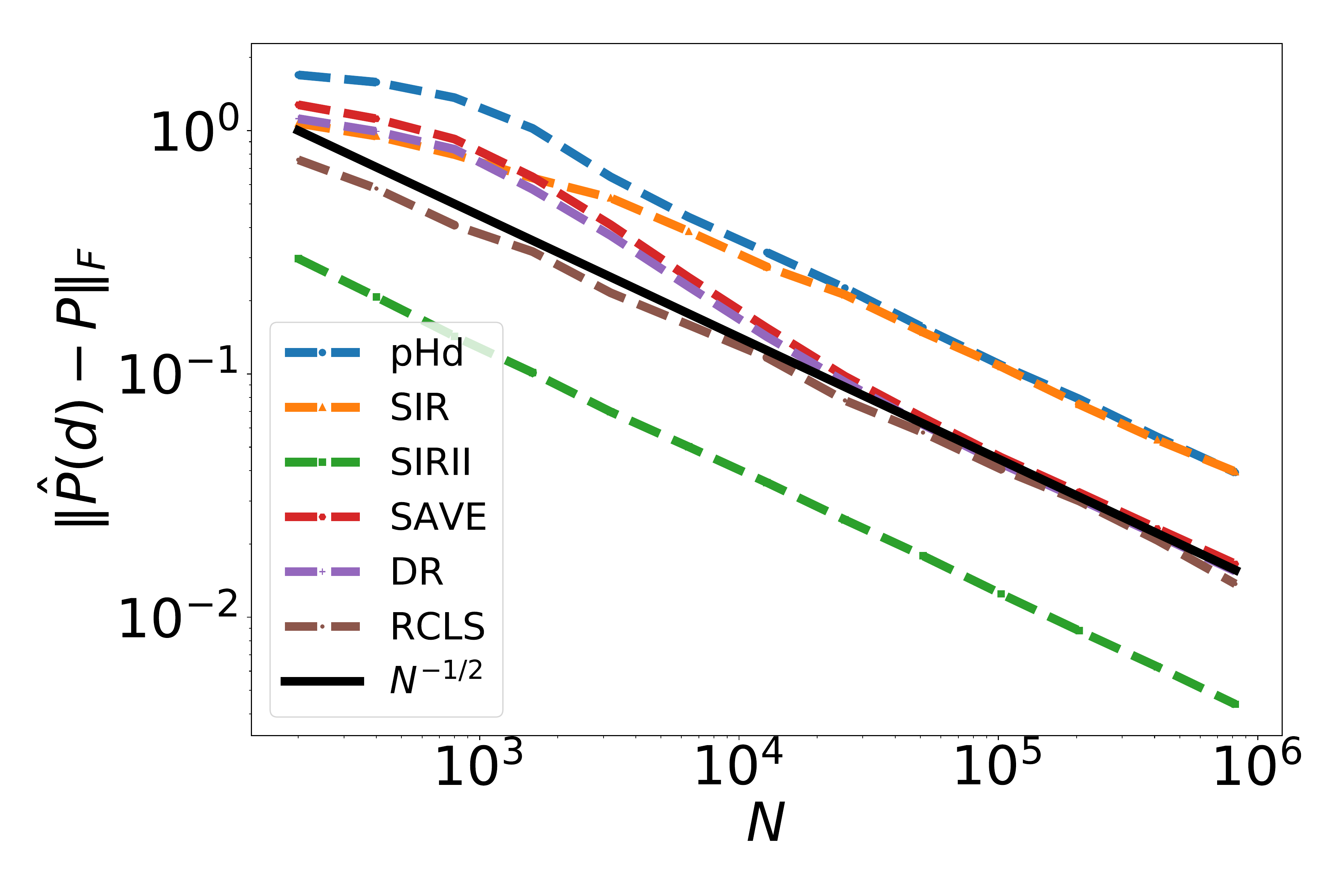}
                \caption{\scriptsize $g(x) = \frac{\sin(x_1) + \cos(x_2 - 1/4)}{(1+x_1^2)}$}
                \label{fig:mim_div_sincos_comparison}
        \end{subfigure}
        \caption{Projection errors $\Vert \MIMPJ{J}(d) - \MIMPTrue\Vert_F$ for different methods as a function of the sample size $N$.
        Except PHD, most estimators achieve an expected $N^{-1/2}$ convergence rate on all problems.
        RCLS and SIRII stand out by consistently achieving first and second best performances.}
\end{figure}

\noindent
We now compare RCLS to the most prominent inverse regression based techniques
SIR, SIRII, SAVE, pHd and DR that have been described extensively in Section \ref{subsec:overview_index_space_estimation}.
In the first part we consider synthetic problems
and we directly assess the performance by evaluating
$\Vert \MIMPJ{J}(d) - \MIMPTrue\Vert_F$, since the true index space is known.
In the second part, we consider real data sets
from the UCI data set repository. Here, the true index space is unknown, and we instead
compare recovered spaces $\Im(\MIMP)$ by measuring the predictive performance of kNN-regression,
when trained on $\{(\MIMP X_i, Y_i) : i \in [N]\}$. In both cases we construct
the partition $\CR_{J,\ell}$ using dyadic cells in the response domain as described in Section \ref{sec:mim_algorithm}.
The source code for all experiments is readily available
at \url{https://github.com/soply/sdr_toolbox} and \url{https://github.com/soply/mim_experiments}.

\subsection{Synthetic data sets}
\label{subsec:synthethic_data_experiments}

\noindent
We sample $X \sim \textrm{Uni}(\{X : \N{X} \leq 1\}) $ in $\bbR^{20}$, and generate the response by $Y = g(A^\top X) + \zeta$
for several functions $g$ and $\zeta \sim \CN(0, 0.01^2\Var(g(A^\top X)))$.
The index space is $A = [e_1|\ldots|e_d] \in \bbR^{D\times d}$
where $e_i$ is the $i$-th standard basis vector.
The hyperparameter $J$ is chosen optimally for SIR, SIRII, SAVE, DR and RCLS
to minimize the projection error within $J \in [100]$. No parameter is required for pHd.

We report projection errors averaged over 100 repetitions of the same experiment
in Figures \ref{fig:mim_simple_division_comparison} - \ref{fig:mim_div_sincos_comparison}. First,
notice that most estimators (except pHd in some cases)
achieve an expected $N^{-1/2}$ rate on all problems. pHd fails to detect linear trends
and therefore fails to detect the index space in some cases.
RCLS achieves the best performance in Figures \ref{fig:mim_simple_division_comparison}-\ref{fig:mim_exp_sinus_comparison},
is tied with SIRII in Figure \ref{fig:mim_exp_sinus_3_comparison},
and runner up to SIRII in the remaining cases.

In Figures \ref{fig:mim_simple_division_comparison}-\ref{fig:mim_exp_sinus_comparison}, where
RCLS improves upon competitors, we observe (temporary) convergence speeds beyond $N^{-1/2}$.
This can be explained by recent results in \cite{acapkarate}, when recognizing that
the multi-index models in \ref{fig:mim_simple_division_comparison}-\ref{fig:mim_exp_sinus_comparison}
have approximately monotone single-index structure if we restrict $(X,Y)$ to
small level sets $Y \in \CR_{J,\ell}$.
More precisely, \cite{acapkarate} shows that the convergence rate of RCLS can
be temporarily as large as $\log(N)N^{-1}$, if the data follows a monotone single-index model.
The reason for this increased convergence speed is that the variance of the local least squares estimator decreases
quadratically in the length of the level set $\CR_{J,\ell}$ \cite{acapkarate,kereta2019nonlinear}, and hence
the convergence speed may exceed $N^{-1/2}$ when choosing $J$ as an increasing function of $N$. These observations
suggest that RCLS is particularly suited for multi-index models where we assume
a response-local single-index structure to be a good fit.

\subsection{Real data sets}
\label{subsec:real_data_experiments}

\begin{table}
  \scriptsize
\begin{center}
\resizebox{\columnwidth}{!}{%
\begin{tabular}{@{}ccc H cc H H cc@{}} \toprule
     Characteristics& {Airquality} & {Ames} & {Auto} & {Boston} & {Concrete} & {EUStock} & {Istanbul} &  {Skillcraft} & Yacht\\ \midrule
    {$\log$-TF} & {No} & {Yes} & {No} & {Yes} & {No} & {Yes} & {No} & {Yes} & {Yes} \\ \midrule
    {$D,\ N$} & $11,\ 7393$& $7,\ 1197$& $4,\ 392$ & $12,\ 506$ & $8,\ 1030$ & $3,\ 1860$ & $7,\ 536$ & $16,\ 3338$ & $6,\ 307$\\ \midrule
    {Factor} & $10^{-1}$ & {$10^{5}$} & $10^{1}$ & $10^{1}$ & $10^{1}$ & $10^3$ & $10^{-2}$ & $10^2$ & $10^1$\\ \midrule
    {$\bar Y\pm\textrm{STD}(Y)$} & $9.95 \pm 4.03$ &  $1.74 \pm 0.67$  & {$2.35\pm 0.78$} & {$1.27 \pm 0.71$} &
    {$3.58\pm1.67$} & {$3.57\pm 0.97$} & {$0.16\pm2.11$} & {$1.15\pm 0.48$} & {$1.05 \pm 1.51$} \\\midrule \\ \\
    {Baselines} & \\ \midrule
    {LinReg} & $1.22\pm0.03$ & $\bf{0.23\pm0.02}$ & $0.44\pm0.05$ & $0.50\pm0.11$ & $1.06\pm0.06$ & $0.18\pm0.01$&$1.38\pm0.13$ & $0.14\pm0.03$ & $0.22\pm0.07$\\ \midrule
    {kNN}  & $1.03\pm0.02$ & $0.26\pm0.03$ & $0.39\pm0.04$ & $0.41\pm0.06$ & $0.89\pm0.08$ & $0.06\pm0.00$& $1.52\pm0.16$ & $0.17\pm0.01$ & $0.76\pm0.11$ \\
    $k$ & {$25.0 $} & {$9.8$} & {$16.3$} & {$6.8$} & {$5.5$} & {$6.0$} & {$17.8$} & {$9.8$} & {$1.1$} \\ \midrule \\ \\
    {SDR + KNN} & \\ \midrule
    {pHd} &$1.01\pm0.04$ & $0.29\pm0.04$ & $0.39+\pm0.04$ & $0.43\pm0.06$ & $0.89\pm0.04$ & $0.06\pm0.00$&$1.46\pm0.13$ & $0.26\pm0.01$ & $0.70\pm0.18$\\
    $d$ & {$10.2$}  & {$6.5$} & {$2.75$}  & {$7.55$} & {$6.65$} & {$3.0$} & {$2.2$} & {$7.0$} & {$5.4$}\\
    $k$ & {$10.0$} & {$7.7$} & {$15.5$}  & {$5.3$} & {$4.0$} & {$5.0$} & {$18.05$} & {$16.4$} & {$1.7$}\\\midrule
    {SIR} &$0.59\pm0.02$ & $0.24\pm0.04$ & $0.39+\pm0.04$ & $0.42\pm0.07$ & $0.87\pm0.06$ & $0.06\pm0.00$&$1.46\pm0.13$ & $0.08\pm0.01$ & $0.18\pm0.07$\\
    $d$ & {$6.15$}  & {$1.4$} & {$2.75$}  & {$4.8$} & {$6.1$} & {$3.0$} & {$2.2$} & {$1.85$} & {$1.15$}\\
    $k$ & {$10.0$} & {$11.4$} & {$15.5$}  & {$6.95$} & {$3.7$} & {$5.0$} & {$18.05$} & {$11.4$} & {$8.4$}\\
    $J$ & {$10.3$} & {$10.8$} & {$5.1$} & {$8.6$} & {$9.8$} & {$5.35$} & {$3.7$} & {$7.65$} & {$8.6$}\\ \midrule
    {SIRII} &$0.87\pm0.02$ & $0.27\pm0.04$ & $0.39\pm0.04$ & $0.44\pm0.07$ & $0.89\pm0.06$ & $0.06\pm0.00$&$1.46\pm0.13$ & $0.29\pm0.03$ & $0.53\pm0.21$\\
    $d$ & {$11.0$}  & {$7.0$} & {$2.75$}  & {$6.65$} & {$8.0$} & {$3.0$} & {$2.2$} & {$2.75$} & {$2.8$}\\
    $k$ & {$10.0$} & {$8.45$} & {$15.5$}  & {$5.3$} & {$4.0$} & {$5.0$} & {$18.05$} & {$16.75$} & {$6.7$}\\
    $J$ & {$12.0$} & {$9.8$} & {$5.1$} & {$8.2$} & {$10.4$} & {$5.35$} & {$3.7$} & {$3.75$} & {$4.6$}\\ \midrule
    {SAVE} &$0.58\pm0.01$ & $0.25\pm0.04$ & $0.39\pm0.07$ & $0.44\pm0.05$ & $0.79\pm0.09$ &$0.06\pm0.01$& $1.51\pm0.18$ & $0.09\pm0.01$ & $0.18\pm0.04$\\
    $d$ & {$6.0$} & {$2.9$} & {$3.1$} & {$7.55$} & {$3.85$} & {$3.0$} & {$7.0$} & {$4.65$} & {$1.7$} \\
    $k$ & {$10.0$} & {$10.7$} & {$13.25$} & {$6.35$} & {$3.2$} & {$5.0$} & {$15.95$} & {$9.5$} & {$7.35$} \\
    $J$ & {$13.4$} & {$12.0$} & {$6.65$} & {$9.9$} & {$8.3$} & {$5.25$} & {$8.0$} & {$8.75$} & {$8.7$} \\ \midrule
    {DR} &$0.58\pm0.02$ & $0.24\pm0.03$ & $0.37\pm0.05$ & $\bf{0.40\pm0.07}$ & $0.75\pm0.08$ &$0.06\pm0.01$& $1.49\pm0.15$ & $0.08\pm0.01$ & $0.18\pm0.06$\\
    $d$ & {$5.85$} & {$2.15$} & {$3.0$} & {$5.45$} & {$3.35$} & {$3.0$} & {$5.15$} & {$2.6$} & {$1.75$} \\
    $k$ & {$10.0$} & {$11.3$} & {$17.0$} & {$6.8$} & {$3.7$} & {$5.0$} & {$13.85$} & {$14.9$} & {$6.55$} \\
    $J$ & {$12.6$} & {$12.0$} & {$7.2$} & {$10.2$} & {$8.5$} & {$6.2$} & {$9.8$} & {$9.3$} & {$8.75$} \\ \midrule
    {RCLS} &$\bf{0.51\pm0.03}$ & $0.25\pm0.03$ & $0.39\pm0.04$ & $0.41\pm0.06$ & $\bf{0.72\pm0.06}$ &$0.06\pm0.00$& $1.46\pm0.14$ & $\bf{0.07\pm0.01}$ & $\bf{0.17\pm0.06}$\\
    $d$ & {$5.4$} & {$1.9$} & {$2.6 $} & {$5.1$} & {$3.25$} & {$3.0$} & {$4.35$} & {$2.95$} & {$1.75$} \\
    $k$ & {$10.0$}  & {$13.55$} & {$20.75$} & {$11.0$} & {$4.3$} & {$5.0$} & {$17.0$} & {$9.15$} & {$5.6$}\\
    $J$ & {$11.9$}  & {$6.4$} & {$5.45$} & {$5.7$} & {$5.7$} & {$6.2$} & {$5.5$} & {$3.2$} & {$7.3$}\\
    \bottomrule
\end{tabular}
}
\end{center}
\caption{RMSE, standard deviation, and cross-validated hyperparameters, over $20$ repetitions for several estimators
and UCI repository data sets. Values for $d$, $k$, $J$ are averages over different runs of each experiment.
First $5$ rows describe the data sets and their characteristics, and the remaining rows contain the results.
For a simplified presentation, we divide the mean and STD of RMSE, and the mean and STD of the data (5th row)
by the value in row \emph{Factor}.}
\label{tab:mim_numerical_experiments}
\end{table}

To compare RCLS with inverse regression based competitors on real data sets, we first compute an index space and then
compare the predictive performance when training a kNN-regressor on projected samples. More precisely,
we conduct the following steps.
\begin{enumerate}
\item Split the data set $\{(X_i,Y_i) : i \in [N]\}$ into training and test set $\CX_\textrm{Train}, \CY_\textrm{Train}$, and $\CX_\textrm{Test}, \CY_\textrm{Test}$
\item Use pHd, SIR, SIRII, SAVE, DR, RCLS on the training set to compute an
index space $\hat A$
\item Train a kNN-regressor using $\{(\hat A^\top X_i, Y_i) : X_i \in \CX_\textrm{Train}\}$
\item Crossvalidate over hyperparameters $d$ (index space dimension), $k$ (kNN parameter), and $J$ (number of level sets) using a hold-out validation set
of the training data
\item Compute the root mean squared error (RMSE) of the kNN-regressor on the test set
\end{enumerate}
\noindent
The test set contains $15\%$ of the data, while cross-validation is performed using a $10$-fold splitting strategy.
Each experiment is repeated 20 times and we report the mean and standard deviation.

We consider the UCI data sets \texttt{Airquality}, \texttt{Ames}-housing, \texttt{Boston}-housing,
\texttt{Concrete}, \texttt{Skillcraft} and \texttt{Yacht}. We standardize the components of $X$ to $[-1,1]$
and potentially perform a $\log$ transformation of $Y$ if the marginal has sparsely populated tails. This is indicated by
the \emph{$\log$-TF} row in Table \ref{tab:mim_numerical_experiments}. For some data sets,
we also exclude features with missing values, or, in the case of \texttt{Ames}, we exclude some
irrelevant and categorical features to reduce the complexity of the data set. Preprocessed
data sets can be found at \url{https://github.com/soply/db_hand}.

The RMSE and cross-validated hyperparameters are presented in Table \ref{tab:mim_numerical_experiments}.
To have robust baselines for comparison, we also compute the RMSE of standard
linear regression and kNN regression. We first see that
applying a dimension reduction technique improves the performance of linear regression
and ordinary kNN significantly on data sets \textit{Airquality}, \textit{Concrete}, \textit{Skillcraft} and \textit{Yacht}.
Furthermore, on these data sets, RCLS convinces by achieving the best
results among all competitors. Runner-up is DR, where SIR and SAVE share third and fourth place.
The results of pHd and SIRII are not convincing on most data sets.

The study confirms that RCLS is a viable alternative to prominent inverse regression methods.
The data sets were chosen because one-dimensional maps $e_i^\top X\mapsto Y$,
where $e_i$ is the $i$-th standard basis vector, show a certain degree of monotonicity. We believe that this promotes
a response-local monotone single-index structure, which is beneficial for
the accuracy of the RCLS approach as briefly discussed in Section \ref{subsec:synthethic_data_experiments}.

\begin{appendix}

\section*{}
\label{sec:app_chap_mim}

\subsection{Probabilistic results}
\label{subsec:app_additional_prob_results}
This section contains some probabilistic auxiliary results used in the paper.
\begin{lemma}
\label{lem:conditional_subgaussianity}
If $Z \in \bbR^D$ is sub-Gaussian and $E$ an event with $\bbP(E) > 0$, then $Z|E$
is sub-Gaussian with $\N{Z|E}_{\psi_2}\leq \N{Z}_{\psi_2}\bbP(E)^{-1}$.
\end{lemma}
\begin{proof}
Assume without loss of generality $Z \in \bbR$. The result for the vector then follows
by the definition.
We use the characterization of sub-Gaussianity by the moment bound in \cite[Proposition 2.5.2, b)]{vershynin2018high}.
So let $p\geq 1$. By the law of total expectation it follows that $\bbE[\SN{Z}^p] = \bbE[\SN{Z}^p|E] \bbP(E) + \bbE[\SN{Z}^p|E^{C}]\bbP(E^C) \geq \bbE[\SN{Z}^p|E] \bbP(E)$.
Dividing $\bbP(E)$ and using monotonicity of the $p$-th root
\begin{align*}
 \left(\bbE[\SN{Z}^p|E]\right)^{1/p} \leq \frac{\left(\bbE[\SN{Z}^p]\right)^{1/p}}{\bbP(E)^{1/p}} \leq \frac{\left(\bbE[\SN{Z}^p]\right)^{1/p}}{\bbP(E)} \leq C\frac{\N{Z}_{\psi_2}\sqrt{p}}{\bbP(E)},
\end{align*}
where $C$ is some universal constant, the second inequality follows from $\bbP(E) \leq 1$, and the third from the sub-Gaussianity of $Z$.
\end{proof}

\begin{lemma}
\label{lem:norm_concentration}
If $X \in \bbR^D$ is sub-Gaussian, so is $\N{X}$,
with $ \| \| X \| \|_{\psi_2} \le \sqrt{D} \|X\|_{\psi_2} $.
\end{lemma}
\begin{proof}
Using H\"older's inequality and the sub-Gaussianity of $X$ we compute
\begin{equation*}
 \bbE \left[ \exp\left( \frac{\| X \|^2}{D \|X\|_{\psi_2}^2} \right) \right] = \bbE \left[ \prod_{i=1}^D \exp\left( \frac{|e_i^TX|^2}{D \|X\|_{\psi_2}^2} \right) \right]
\le \left( \prod_{i=1}^D \bbE \left[ \exp\left( \frac{|e_i^TX|^2}{\|X\|_{\psi_2}^2} \right) \right] \right)^{1/D} \le 2. \qedhere
\end{equation*}
\end{proof}

\begin{lemma}
\label{lem:bound_rho}
Fix $u > 0$, $\varepsilon > 0$. Let $Y \in \bbR$ be a random variable, $\CR$ an interval, and
$\hat \bbP(Y \in \CR) := \SN{\{Y_i \in \CR\}}N^{-1}$ the empirical estimate of $\bbP(Y \in \CR)$
based on $N$ iid. samples. Then
\begin{align}
\label{eq:emp_density_hoeffding}
&\bbP \left(\SN{\hat \bbP(Y \in \CR) - \bbP(Y \in \CR)} \leq \varepsilon \right) \geq 1 - \exp(-u)  \hspace{3pt} \textrm{ if } N > Cu\varepsilon^{-2}, \\
\label{eq:emp_density_chernoff}
&\bbP \left( \SN{\hat \bbP(Y \in \CR) - \bbP(Y \in \CR)}\leq \frac{1}{2}\bbP(Y \in \CR)   \right) \geq 1 - 2\exp(-u) \hspace{3pt} \textrm{ if } N > C\frac{u}{\bbP(Y \in \CR)}.
\end{align}
\end{lemma}
\begin{proof}
For \eqref{eq:emp_density_hoeffding}, we write
$\hat \bbP(Y \in \CR) - \bbP(Y \in \CR)$ as a sum of iid. centred random variables
$Z_i := 1_{\CR}(Y_i) - \bbP(Y \in \CR)$ that are
bounded by $\SN{Z_i}\leq 1$. The result follows by applying Hoeffding's inequality
\cite[Theorem 2.2.6]{vershynin2018high}.
For the Chernoff-type bound \eqref{eq:emp_density_chernoff} we recognize $\hat \bbP(Y \in \CR)N$ as a sum of independent
random variables $1\{Y_i \in \CR\}$ with values in $\{0,1\}$, and $\bbE[\hat \bbP(Y \in \CR)N] = \bbP(Y \in \CR)N$.
\cite{mitzenmacher2017probability} provides the bound
\[
\bbP \left( \SN{\hat \bbP(Y \in \CR) N - \bbP(Y \in \CR) N}\geq \frac{1}{2}\bbP(Y \in \CR) N\right) \leq 2\exp\left(-\frac{\bbP(Y \in \CR)N}{12} \right),
\]
and the result follows division by $N$, and  $N > 12\bbP(Y \in \CR)^{-1}u$.
\end{proof}

\subsection{Differences of projections}
\label{subsec:app_mim_index_space_estimation}
We gather two auxiliary results to rewrite the norm of differences of projections.

\begin{lemma}
\label{lem:projection_equality}
Let $A$ and $B$ be subspaces with $\dim(A) = \dim(B)$, and let
$P_A$ and $P_B$ the corresponding orthogonal projections. For $P_{A^\perp} = \Id - P_A$ we get
$\N{P_{A} - P_{B}} = \N{P_{A^\perp}P_{B}}.
$
\end{lemma}
\begin{proof}
Assume $\N{\left(\Id - P_{A}\right)P_{B}} = \N{P_{A^\perp}P_{B}} < 1$ first. Then
the first case of Theorem 6.34 in Chapter 1 in \cite{kato2013perturbation} applies.
Note that the second case can be ruled out since $P_{A}$ can not map $\Range{P_B}$
one-to-one onto a proper subspace of $V \subset \Range{P_A}$
because $\dim(V) < \dim(\Range{P_A}) = \dim(\Range{P_B})$ according to the assumption.
Thus, in the first case it follows that
\[
\N{\left(\Id - P_{A}\right)P_{B}} = \N{\left(\Id - P_{B}\right)P_{A}} = \N{P_{A} - P_{B}}.
\]
Now let $\N{P_{A^\perp}P_{B}} = 1$. Then there exists $v \in \bbS^{D}$ such that
$\N{P_{A^\perp}P_{B}v} = \N{v}$. Since
\[
\N{v} \geq \N{P_{B}v} \geq \N{P_{A^\perp}P_{B}v} = \N{v},
\]
it follows that $\N{P_{B}v} = \N{v}$, and thus $P_Bv = v$ because $P_B$ is a
projection. With the same argument we deduce also $P_{A^\perp}v = v$, and then
\[
\left(P_{A} - P_{B}\right)v = P_{A}v - P_{B}v = 0 - v = v
\]
implies $\N{P_{A} - P_{B}} = 1 = \N{P_{A^\perp}P_{B}}$.
\end{proof}

\begin{lemma}
\label{lem:projection_equality_frobenius}
Let $A$ and $B$ be subspaces with $m = \dim(A) = \dim(B)$, and let
$P_A$ and $P_B$ the corresponding orthogonal projections. For $P_{A^\perp} = \Id - P_A$ we get
$\N{P_{A} - P_{B}}_F = \sqrt{2}\N{P_{A^\perp}P_{B}}_F$.
\end{lemma}
\begin{proof}
With slight abuse of notation we denote $A, B\in \bbR^{D\times m}$ two orthonormal bases
of $A$ respectively $B$ such that $P_A = AA^\top$, $P_B = BB^\top$. Now, denote
$A^\top B = U(cos(\theta))V^\top$ where $\cos(\theta) \in \bbR^{m\times m}$ is the diagonal matrix
containing the principal angles $\theta_i$ \cite{hamm2008grassmann}. From \cite{hamm2008grassmann}
we obtain the identity $1/2 \N{P_A - P_B}_F^2 = m  - \sum_{i=1}^{m}\cos^2(\theta_i)$. Doing some
further manipulations we get
\begin{align*}
\frac{1}{2} \N{P_A - P_B}_F^2 &= m  - \sum_{i=1}^{m}\cos^2(\theta_i) = m - \N{A^\top B}_F^2
\\
&= m - \Trace(B^\top A A^\top B) = m - \Trace(A A^\top B B^\top ) \\
&= m - \Trace(P_A P_B) = m - \Trace((\Id-P_{A^\perp}) P_B) \\
&= m - \Trace(P_B) + \Trace(P_{A^\perp}P_B).
\end{align*}
The result follows by $\Trace(P_B) = \dim(B) = m$ and $\Trace(P_{A^\perp}P_B) = \N{P_{A^{\perp}}P_B}_F^2$.
\end{proof}

\subsection{Proof of Theorem \ref{thm:regression_error_kNN}}
\label{subsec:app_kNN_mim_proof}
\begin{proof}[Proof of Theorem \ref{thm:regression_error_kNN}]
Denote $S_X = \{X_i : i \in [N]\}$ and
$\tilde f_k(x) = \bbE[\hat f_k(x) |S_X] = \sum_{i=1}^{k} f(X_{i(x)})$ for fixed $x$.
We first decompose (randomness is in the $\zeta_{i}$'s)
\begin{align*}
\bbE\left[\left(\hat f_k(x) - f(x)\right)^2 \Big|S_X\right]= \bbE\left[\left(\hat f_k(x) - \tilde f_k(x)\right)^2 \Big| S_X\right] +
\left(\tilde f_k(x) - f(x)\right)^2,
\end{align*}
and then use the towering property of conditional expectations to obtain
\begin{align*}
\bbE \left(\hat f_k(X) - f(X)\right)^2 &= \bbE \bbE\left[ \left(\hat f_k(X) - f(X)\right)^2\Big|\CS_X\right]\\
&=\bbE\bbE \left[\left(\hat f_k(X) - \tilde f_k(X)\right)^2\Big|\CS_X\right]
+ \bbE\left(\tilde f_k(X) - f(X)\right)^2 \\
&=\bbE \left(\hat f_k(X) - \tilde f_k(X)\right)^2 + \bbE\left(\tilde f_k(X) - f(X)\right)^2
\end{align*}
Since $\bbE \zeta_{i} = 0$, $\zeta_{i}\independent \zeta_j$ and $\textrm{Var}(\zeta_i| X = x) \leq \sigma_{\zeta}^2$, the first term
satisfies the bound
\begin{align*}
\bbE\left(\hat f_k(X) - \tilde f_k(X)\right)^2
= \textrm{Var}\left(\frac{1}{k}\sum\limits_{i=1}^{k}\zeta_{i(X)}\right)
\leq \frac{\sigma_{\zeta}^2}{k} = C_k^{-1} \sigma_{\zeta}^2 N^{-\frac{2s}{2s+d}}.
\end{align*}
For $\bbE\left(\tilde f_k(X) - f(X)\right)^2$, we recall that $f(X) = g(A^\top X)$ for some $(L,s)$-smooth $g$, which implies
\begin{align*}
\bbE \left(\tilde f_k(X) - f(X)\right)^2 &= \bbE\left(\frac{1}{k}\sum\limits_{i=1}^{k}g(A^\top X_{i(X)}) - g(A^\top X)\right)^2
\\
&\leq 4 L^2 \bbE\left(\frac{1}{k}\sum\limits_{i=1}^{k}\min\left\{\N{A^\top \left(X_{i(X)} - X\right)}^{s}, 1\right\}\right)^2,
\end{align*}
where the $4$ can be injected since $\SN{f(X)}\leq 1$ almost surely.
To bound this term further we have to replace $\{X_{i(X)} : i \in [k]\}$ (the $k$ closest samples wrt to
$\hat d(\cdot,X) := \Vert \MIMP\left(\cdot - X\right)\Vert$) by the $k$ closest samples based on $d(\cdot, X) := \Vert \MIMPTrue\left(\cdot - X\right)\Vert$.
So let $\tilde X_{i(X)}$ denote the $i$-th closest sample to $X$ based on $d$, and let further $\delta := \N{\MIMP - \MIMPTrue}$. Since
\begin{align*}
\SN{d(X, X') - \hat d(X, X')}
\leq \N{(\MIMPTrue - \MIMP)(X - X')} \leq \delta \N{X - X'},
\end{align*}
and $(a+b)^s \leq a^s + b^s$ for $s \leq 1$, we can bound
\begin{align*}
\sum\limits_{i=1}^{k}\min\{d(X_{i(X)}, X)^{s}, 1\}
&\leq \sum\limits_{i=1}^{k}\min\{\hat d(X_{i(X)}, X)^{s} + \delta^s\max_{i \in [N]}\N{X_i-X}^s, 1\}\\
&\leq \sum\limits_{i=1}^{k}\min\{\hat d(\tilde X_{i(X)}, X)^{s} + \delta^s\max_{i \in [N]}\N{X_i-X}^s, 1\}\\
&\leq \sum\limits_{i=1}^{k}\min\left\{d(\tilde X_{i(X)}, X)^{s} + 2\delta^s \max_{i \in [N]}\N{X_i-X}^s, 1\right\},
\end{align*}
where in the second inequality we used that $X_{i(X)}$ minimizes the distance to $X$ measured in $\hat d$, and can therefore
be replaced by $\tilde X_{i(X)}$. Denoting $\Delta_{X,N} := 2\delta^s\max_{i \in [N]}\N{X_i-X}^s$
and using $(\sum_{i=1}^{k}b_i)^2 \leq k \sum_{i=1}^{k}b_i^2$ for arbitrary $b_i$'s, we get
\begin{equation}
\begin{aligned}
\label{eq:aux_kNN_bound_1}
\bbE\left(\hat f_k(X) - f(X)\right)^2
&\leq 4 L^2 \bbE \left(\frac{1}{k}\sum\limits_{i=1}^{k}\min\{d(\tilde X_{i(X)}, X)^s + \Delta_{X,N}, 1\}\right)^2\\
&\leq \frac{4 L^2 }{k}\sum\limits_{i=1}^{k}\bbE \min\{(d(\tilde X_{i(X)}, X)^{s} + \Delta_{X,N})^2, 1\}\\
&\leq \frac{8 L^2 }{k}\sum\limits_{i=1}^{k}\bbE\min\{d(\tilde X_{i(X)}, X)^{2s}, 1\} +
\frac{8 L^2 }{k}\sum\limits_{i=1}^{k}\bbE\Delta_{X,N}^2.
\end{aligned}
\end{equation}
For the first term, we proceed as in \cite{gyorfi2006distribution,kohler2006rates} by randomly splitting the data set
$\{X_i : i \in [N]\}$ into $k + 1$ sets, where the first $k$ sets contain $\lfloor N/k\rfloor$ samples. Then we let
$X^*_{i(X)}$ denote the nearest neighbor to $X$ (measured in $d$)  within the $i$-th set.
Since $\{\tilde X_{i(X)} : i \in [k]\}$ are by definition the closest $k$ samples (measured in $d$), we can
bound
\begin{align*}
\frac{8L^2 }{k}\sum\limits_{i=1}^{k}\bbE\min\{d(\tilde X_{i(X)}, X)^{2s}, 1\} &\leq
\frac{8L^2 }{k}\sum\limits_{i=1}^{k}\bbE\min\{d(X^*_{i(X)}, X)^{2s}, 1\}\\
&=8 L^2\bbE\min\left\{\N{\MIMPTrue\left(X^*_{1(X)} - X\right)}^{2s}, 1\right\},
\end{align*}
where the last equality uses that the distribution of $X^*_{i(X)} - X$ is independent of the set index $i$.
Since $\N{\MIMPTrue(\cdot)} = \N{A^\top (\cdot)}$, $d > 2s$ by assumption, and
\[
\bbE\N{A^\top X}^{\beta} = \bbE \N{X}^{\beta} \lesssim \N{X}_{\psi_2}^{\beta}\beta^{\beta/2}
\]
for any $\beta \geq 1$ by the sub-Gaussianity of $X$ (see \cite[Proposition 2.5.2]{vershynin2018high}),
Lemma 1 in \cite{kohler2006rates} implies the existence of a constant $C_1 = C_1(d,s,\N{X}_{\psi_2})$
satisfying
\begin{align*}
\bbE\min\left\{\N{\MIMPTrue\left(X^*_{1(X)} - X\right)}^{2s}, 1\right\} &= \bbE\min\left\{\N{A^\top\left(X^*_{1(X)} - X\right)}^{2s}, 1\right\} \\
&\leq C_1 \left(\frac{k}{N}\right)^{\frac{2s}{d}} = C_1 C_k^{\frac{2s}{d}} N^{-\frac{2s}{2s+d}}.
\end{align*}
It remains to bound the last term in \eqref{eq:aux_kNN_bound_1}. Denote for short $\sigma_X = \N{X}_{\psi_2}$. We first compute that
\begin{align*}
\bbE\Delta_{X,N}^2 &= \int_{0}^{\infty}\bbP\left(\Delta_{X,N}^2 > u\right)du =\int_{0}^{\infty}\bbP\left(\max_{i=1,\ldots,N}\N{X_i - X}^{2s} > \frac{u}{4\delta^{2s}}\right)du.
\end{align*}
We can control this probability by using the sub-Gaussianity of $X$. More precisely,
since $X$ is sub-Gaussian, $X_i - X$ is sub-Gaussian (norm changing only by a universal constant),
and Lemma \ref{lem:norm_concentration} implies that $\Vert \N{X_i - X}\Vert_{\psi_2} \leq C\sqrt{D}\sigma_X$. Taking the square, and using
\cite[Lemma 2.7.6]{vershynin2018high}, we obtain
\[
\Vert \N{X_i -X}^2\Vert_{\psi_1} \leq \Vert \N{X_i -X}\Vert_{\psi_2}^2 \leq  C D \sigma_X^2.
\]
To bound the integral, we first split $[0,\infty]$ into $[0, \nu D\sigma_X^2\log(N)\delta^{2s}]$
and $[\nu D\sigma_X^2\log(N)\delta^{2s}, \infty]$ for some $\nu > 4\max\{\frac{1}{D\sigma_X^2\log(N)}, C\}$, which yields
\begin{align*}
\bbE\Delta_{X,N}^2 &\leq \int_{\nu D\sigma_X^2\log(N)\delta^{2s}}^{\infty}\bbP\left(\max_{i \in [N]}\N{X_i - X}^{2s} > \frac{u}{4\delta^{2s}}\right)du
+ \nu D\sigma_X^2\log(N)\delta^{2s}.
\end{align*}
For the first term we realize that $u > \nu D\sigma_X^2\log(N)\delta^{2s} > 4\delta^{2s}$ implies
\begin{align*}
&\bbP\left(\max_{i=1,\ldots,N}\N{X_i - X}^{2s} > \frac{u}{4\delta^{2s}}\right) \leq
\bbP\left(\max_{i=1,\ldots,N}\N{X_i - X}^{2} > \frac{u}{4\delta^{2s}}\right).
\end{align*}
Then the sub-Exponentiality of $\Vert X-X_i\Vert^2$ and a union bound argument over $i \in [N]$ give
\begin{align*}
\bbE\Delta_{X,N}^2 &\leq \int_{\nu D\sigma_X^2\log(N)\delta^{2s}}^{\infty}\bbP\left(\max_{i \in [N]}\N{X_i - X}^{2} > \frac{u}{4\delta^{2s}}\right)du
+ \nu D\sigma_X^2\log(N)\delta^{2s}\\
&\leq  2N \int_{\nu D\sigma_X^2\log(N)\delta^{2s}}^{\infty}\exp\left(-\frac{u}{CD\sigma_X^2\delta^{2s}}\right)du
+ \nu 2D\sigma_X^2\log(N)\delta^{2s}\\
&\leq 2CD\sigma_X^2\delta^{2s} N \exp\left(-\frac{\nu \log(N)}{C}\right) + \nu D\sigma_X^2\log(N)\delta^{2s}\\
&\leq 2CD\sigma_X^2\delta^{2s} N \exp\left(-\log(N)\right) + \nu D\sigma_X^2\log(N)\delta^{2s}\\
&\leq 2(\nu \vee C)  D \sigma_X^2\log(N) \delta^{2s}. \qedhere
\end{align*}
\end{proof}

\subsection{Proof of Theorem \ref{thm:regression_error_piecewise_pols}}
\label{subsec:mim_pp_proof}
\paragraph*{Interlude: Smoothness of linear concatenations}
In this section we establish smoothness properties of linear concatenations
with explicit bounds for corresponding Lipschitz constants.
\begin{lemma}
\label{lem:derivative_formula}
Let $\phi : \bbR^{d}\rightarrow \bbR$, $W \in \bbR^{d \times d}$ and $\psi(z) = \phi(Wz)$.
Let $s \in \bbN$ and $\alpha \in \bbN^{d}$ be a multi-index with $\sum_{i=1}^{d} \alpha_i = k \leq s$. If $\phi \in \CC^{s}(\bbR^{d})$, i.e.
all partial derivatives $\partial^{\alpha}\phi$ exist and are continuous, then also all $\partial^{\alpha}\psi$ exist and are
continuous. Moreover, if $i :[k] \rightarrow [d]$ is an arbitrary derivative ordering satisfying $\alpha = \sum_{w=1}^{k}e_{i(w)}$, we can express for any $k \in [s]$
\begin{align}
\label{eq:derivative_formula}
\partial^{\alpha}\psi(z) = \sum_{i_1 = 1}^{d}\cdots \sum_{i_k=1}^{d}\left(\prod\limits_{w = 1}^{k} W_{i_{w}, i(w)}\right) \partial_{i_1}\cdots\partial_{i_k}(\phi)(Wz).
\end{align}
\end{lemma}
\begin{example}
Let $\alpha = e_i + e_j$ and $i(1) = i$,  $i(2) = j$. Then the formula yields the derivative
\begin{align*}
\partial^{\alpha}\psi(z) = \sum_{i_1=1}^{d}\sum\limits_{i_2=1}^{d}W_{i_1,i}W_{i_2,j} \partial_{i_1}\partial_{i_2}(\phi) (Wz).
\end{align*}
\end{example}
\begin{proof}
$\psi$ is a concatenation of a $\CC^{s}$ function with a linear transformation and is therefore as smooth as $\phi$. For the formula, we
use induction over $k$. Let $\alpha$ be a multi-index with $\sum_{i=1}^{d} \alpha_i = 1$, i.e. $\alpha$ is equal to a
standard basis vector $e_i$ for some $i \in [d]$. Since $\nabla \psi(z) = W^T \nabla \phi(Wz)$ we have
\begin{align*}
\partial_i \psi(z) =\left\langle W_i, \nabla \phi(Wz)\right\rangle  = \sum_{i_1=1}^{d}W_{i_1,i}\partial_{i_1}(\phi)(Wz).
\end{align*}
For the induction step $k-1\rightarrow k$, we let $\alpha$ be a multi-index with $\sum_{i=1}^{d} \alpha_i = k$ and
we calculate $ \partial^{\alpha}\psi(z) = \partial_{i(k)}\partial^{\alpha-e_{i(k)}}\psi(z) $.
Since $\alpha-e_{i(k)}$ is a multi-index whose entries sum to $k-1$,
by induction hypothesis we have
\begin{align*}
\partial^{\alpha}\psi(z) &=
\partial_{i(k)}\left( \sum_{i_1 = 1}^{d}\ldots \sum_{i_{k-1}=1}^{d}\left(\prod\limits_{w = 1}^{k-1} W_{i_{w}, i(w)}\right) \partial_{i_1}\cdots\partial_{i_{k-1}}(\phi)(Wz)\right)\\
&= \sum_{i_1 = 1}^{d}\cdots \sum_{i_{k-1}=1}^{d}\left(\prod\limits_{w = 1}^{k-1} W_{i_{w}, i(w)}\right) \partial_{i(k)}\left(\partial_{i_1}\cdots\partial_{i_{k-1}}(\phi)(Wz)\right)\\
&= \sum_{i_1 = 1}^{d}\cdots \sum_{i_{k-1}=1}^{d}\left(\prod\limits_{w = 1}^{k-1} W_{i_{w}, i(w)}\right) \sum_{i_k=1}^{d} W_{i_k, i(k)}\partial_{i_1}\cdots\partial_{i_{k}}(\phi)(Wz)\\
&= \sum_{i_1 = 1}^{d}\cdots \sum_{i_{k-1}=1}^{d}\sum_{i_k=1}^{d} \left(\prod\limits_{w = 1}^{k-1} W_{i_{w}, i(w)}\right)W_{i_k, i(k)}  \partial_{i_1}\cdots\partial_{i_{k}}(\phi)(Wz),
\end{align*}
where we used Schwartz Lemma in the second to last equality
The result follows by extending the product.
\end{proof}
\begin{lemma}
\label{lem:smooth_transformed_fun}
Let $\phi : \bbR^d \rightarrow \bbR$, $s_1 \in \bbN_{0}$, $0 < s_2 \leq 1$ and $s=s_1+s_2$. Assume $\phi$ is $(L,s)$-smooth, $W \in \bbR^{d\times d}$, and define $\psi(z) = \phi(Wz)$ for some $W \in \bbR^{d\times d}$.
Then $\psi$ is $(Ld^{\frac{s_1}{2}}\N{W}^{s}, s)$-smooth.
\end{lemma}
\begin{proof}
Since $W$ is a linear transformation, $\psi$ has as many continuous partial derivatives as $\phi$.
Now consider $\alpha \in \bbN_0^{d}$ with $\sum_{i=1}^{d}\alpha_i = s_1$, and let
$i:[s_1]\rightarrow [d]$ be an arbitrary derivative ordering satisfying $\sum_{w=1}^{s_1}e_{i(w)} = \alpha$.
By using Lemma \ref{lem:derivative_formula}, we get
\begin{equation*}
\begin{aligned}
& \SN{\partial^{\alpha} \psi(z) - \partial^{\alpha} \psi(z')} \\
= &\SN{\sum_{i_1 = 1}^{d}\cdots \sum_{i_{s_1}=1}^{d}\left(\prod\limits_{w = 1}^{s_1}W_{i_{w}, i(w)}\right)\left(\partial_{i_1}\cdots\partial_{i_{s_1}}(\phi)(Wz)- \partial_{i_1}\cdots\partial_{i_{s_1}}(\phi)(Wz')\right)} \\
\leq & \max_{\alpha : \sum_{i=1}^{d}\alpha_i = s_1} \SN{\partial^{\alpha}\phi(Wz) - \partial^{\alpha}\phi(Wz')} \left(\sum_{i_1 = 1}^{d}\cdots \sum_{i_{s_1}=1}^{d}\prod\limits_{w = 1}^{s_1} \SN{W_{i_{w}, i(w)}}\right).
\end{aligned}
\end{equation*}
Furthermore denote $\N{W}_{1} = \max_{i}\sum_j|W_{i,j}| \leq \sqrt{d} \N{W}$. Then we can rewrite
\begin{align*}
\sum_{i_1 = 1}^{d}\cdots \sum_{i_{s_1}=1}^{d}\prod\limits_{w = 1}^{s_1} \SN{W_{i_{w}, i(w)}}
=\sum_{i_1 = 1}^{d} \SN{W_{i_1,i(1)}}\cdots \sum_{i_{s_1}=1}^{d} \SN{W_{i_{s_1}, i(s_1)}}
\leq \N{W}_{1}^{s_1} \leq d^{\frac{s_1}{2}}\N{W}^{s_1}.
\end{align*}
Combining this with the previous calculation, and the fact that $\phi$ is $(L,s)$-smooth, we get
\begin{equation*}
\SN{\partial^{\alpha} \psi(z) - \partial^{\alpha} \psi(z')} \leq L \N{Wz - Wz'}^{s_2} d^{\frac{s_1}{2}}\N{W}^{s_1}
\leq L d^{\frac{s_1}{2}}\N{W}^{s}\N{z - z'}^{s_2}. \qedhere
\end{equation*}
\end{proof}

\paragraph*{Bounding $\dim(\CF(\hat A, l,k,R))$}
\begin{lemma}
\label{lem:bound_number_of_cells}
We have $\SN{\Delta_l(R)} \leq \lceil(2^{l+1}R)^d\rceil$, and thus
$$ \dim(\CF(\hat A, l,k,R)) \leq {{d+k}\choose k}\lceil(2^{l+1}R)^d\rceil . $$
\end{lemma}
\begin{proof}
First we note that the number of cells with side length $2^{-l}$ required to cover $[-R,R]^d$ is given by $\lceil (2R 2^l)^d\rceil = \lceil(2^{l+1}R)^d\rceil$.
Furthermore, for any $w \in \{\hat A^\top z : z \in B_R(0)\}$, we have
$\N{w} = \N{\hat A^\top z} \leq \N{z} \leq R$, hence $w \in B_{R}(0)$ (in $\bbR^d$). Therefore a bound for $\SN{\Delta_l(R)}$
is given by a bound for the number of cells covering $[-R,R]^d$.
\end{proof}

\paragraph*{Bounding the approximation error}
We first show the existence of $g^*$ almost as regular as $g$
and satisfying $g^*(\hat A^\top x) \approx g(A^\top x)$. Then we bound the approximation
error between $f$ and $h$ over $B_R$.
Finally, we provide the bound for the mean
squared approximation error (second term in \eqref{eq:combined_bound}).
\begin{lemma}
\label{lem:close_function}
Let $g$ be $(L, s)$-smooth with $s=s_1+s_2$, $s_1 \in \bbN_{0}$, $s_2 \in (0,1]$, and $\Vert \MIMP - \MIMPTrue\Vert < 1$.
Then $(\hat A^\top A)^{-1}$ exists, and the function
$g^*(z) := g((\hat A^\top A)^{-1}z)$ is $(L^*, s)$-smooth for $L^* := Ld^{s_1/2}(1-\Vert \MIMP - \MIMPTrue\Vert^2)^{-s/2}$. Moreover, it achieves
\begin{align}
\label{eq:close_function_guarantee}
\SN{g^*(\hat A^\top x) - g(A^\top x)} \leq L^* \N{x}^{1 \wedge s}\Vert \MIMP - \MIMPTrue\Vert^{1 \wedge s}.
\end{align}
\end{lemma}
\begin{proof}
Let $\delta:=\Vert \MIMP - \MIMPTrue\Vert < 1$, and denote the singular value decomposition $\hat A^\top A = USV^\top$, where $S$ denotes the cosines of
principal angles between $\Im(\hat A)$ and $\Im(A)$ in descending order.
It is known from \cite[Definition 2 and Equation (5)]{hamm2008grassmann} that $\delta  = (1-S_{dd}^2)^{1/2}$, which implies
$1 \geq \N{S} \geq \sqrt{1 - \delta^2}$,
hence $\hat A^\top A$ is invertible with $\Vert (\hat A^\top A)^{-1}\Vert \leq (1-\delta^{2})^{-1/2}$.
Applying Lemma \ref{lem:smooth_transformed_fun},
$g^*$ is $(L^*,s)$-smooth. Furthermore we have
$g^*(\hat A^\top A A^\top x) = g((\hat A^\top A)^{-1}\hat A^\top A A^\top x) = g(A^\top x)$.
Using the smoothness of $g^*$, it follows that
\begin{align*}
\SN{g^*(\hat A^\top x) - g(A^\top x)} &= \SN{g^*(\hat A^\top x) - g^*(\hat A^\top A A^\top x)} \leq  L^* \N{\hat A^\top \left(\Id_D - P \right) x}^{1 \wedge s}\\
&= L^* \N{\hat A^\top Q x}^{1 \wedge s} \leq L^* \N{\MIMP Q}^{1 \wedge s}\N{x}^{1 \wedge s}
\\
&= L^*\N{\MIMP - \MIMPTrue}^{1 \wedge s}\N{x}^{1 \wedge s},
\end{align*}
where we used Lemma \ref{lem:projection_equality} in the last equality.
\end{proof}
\begin{proposition}
\label{prop:close_function_h}
Let $f(x) = g(A^\top x)$ for $\MIMPTrue = AA^\top$, $g$ be $(L, s)$-smooth with $s=s_1+s_2$, $s_1 \in \bbN_{0}$, $s_2 \in (0,1]$, and $\Vert \MIMP - \MIMPTrue\Vert < 1$.
There exists a function $h \in \CF(\hat A, l, s_1, R)$ such that
\begin{align}
\label{eq:mim_app_aux_approximation_error}
\max_{x \in B_R}\SN{h(x) - f(x)} \leq L^* \frac{d^{s_1+\frac{s_2}{2}}}{s_1!} 2^{-s(l+1)} + L^* R^{1 \wedge s} \N{\MIMP-\MIMPTrue}^{1 \wedge s},
\end{align}
where $L^* := Ld^{s_1/2}(1-\Vert \MIMP - \MIMPTrue\Vert^2)^{-s/2}$.
\end{proposition}
\begin{proof}
First notice that  $\SN{h(x) - f(x)} \leq \SN{h(x) - f^*(x)} + \SN{f^*(x) - f(x)}$, where
$f^*(x) := g^*(\hat A^\top x)$ is the function defined in Lemma \ref{lem:close_function}.
Using the bound in Lemma \ref{lem:close_function}, and $\N{x} \leq R$,
the second term is bounded by $L^* R^{1\wedge s}\Vert \MIMP - \MIMPTrue\Vert^{1\wedge s}$.
It remains to bound $\SN{h(x) - f^*(x)}$ for a suitably chosen $h$. Since $f^*(x) = g^*(\hat A^\top x)$
and $g^*$ is $(L^*, s)$-smooth, we can use the multivariate Taylor theorem to expand $g^*$ as
\begin{align}
\label{eq:taylor_representation}
g^*(z) = \sum_{\SN{\alpha} \leq s_1-1} \frac{\partial^{\alpha}g^*(z_0)}{\alpha!}(z - z_0)^{\alpha} + \sum_{\SN{\alpha} = s_1}\frac{\partial^{\alpha}g^*(z_0)}{\alpha!}\eta^{\alpha}
\end{align}
for some $\eta$ on the line segment from $z$ to $z_0$. We define the
function $h$ as follows: for a cell
$c \in \Delta_l$, let $z_c \in \bbR^{d}$ denote the center point of the cell, and set $h_c$
to
\begin{align*}
h_c(z) := \sum\limits_{\SN{\alpha} \leq s_1}\frac{\partial^{\alpha}g^*(z_c)}{\alpha!}(z - z_c)^{\alpha}.
\end{align*}
Then we define $h \in \CF(\hat A, l,s_1, R)$ by
\begin{align*}
h(x) := 1_{B_R(0)}(x) \sum\limits_{c \in \Delta_l(R)} 1_{c}(\hat A^\top x)h_c(\hat A^\top x)
= 1_{B_R(0)}(x)h_{c(x)}(\hat A^\top x),
\end{align*}
where $c(x) := \{ c \in \Delta_l(R) : x \in c \}$.
To prove \eqref{eq:mim_app_aux_approximation_error}, we now use \eqref{eq:taylor_representation}
with $z_0 = z_{c(x)}$ and compute
\begin{align*}
h(x) - g^*(\hat A^\top x) &=
\sum\limits_{\SN{\alpha} \leq s_1}\frac{\partial^{\alpha}g^*(z_{c(x)})}{\alpha!}\left(\hat A^\top x - z_{c(x)}\right)^{\alpha} - g^*(\hat A^\top x)\\
&=\sum\limits_{\SN{\alpha} = s_1}\frac{\partial^{\alpha}g^*(z_{c(x)})-\partial^{\alpha}g^*(\eta)}{\alpha!}\left(\hat A^\top x - z_{c(x)}\right)^{\alpha},
\end{align*}
where $\eta$ lies on the line between $\hat A^\top x$ and $z_{c(x)}$. The smoothness of $g^*$ implies
\begin{align*}
\SN{h(x) - g^*(\hat A^\top x)}
&\leq \sum\limits_{\SN{\alpha} = s_1}\frac{\SN{\partial^{\alpha}g^*(z_{c(x)}) - \partial^{\alpha}g^*(\eta)}}{\alpha!}\SN{\left(\hat A^\top x - z_{c(x)}\right)^{\alpha}}\\
&\leq \sum\limits_{\SN{\alpha} = s_1}\frac{L^*\N{z_{c(x)}-\eta}^{s_2}}{\alpha!}\SN{\left(\hat A^\top x - z_{c(x)}\right)^{\alpha}}.
\end{align*}
Since $\hat A^\top x, z_{c(x)} \in c(x)$, we can furthermore bound
\begin{align*}
\SN{\left(\hat A^\top x - z_{c(x)}\right)^{\alpha}} = \SN{\prod\limits_{i=1}^{d}( (\hat A^\top x)_i - (z_{c(x)})_i)^{\alpha_i}}
\leq \prod\limits_{i=1}^{d}\left(2^{-(l+1)}\right)^{\alpha_i} = 2^{-(l+1)s_1}.
\end{align*}
Furthermore since $c(x)$ is convex, and $\eta$ is on the line between $A^\top x$ and $c(x)$,
it follows that $\eta \in c(x)$ and therefore also $\N{z_{c(x)}-\eta} \leq 2^{-(l+1)}\sqrt{d}$. Thus
\begin{align*}
& \SN{h(x) - g^*(\hat A^\top x)} \leq L^* d^{\frac{s_2}{2}}2^{-(l+1)s}\sum\limits_{\SN{\alpha} = s_1}\frac{1}{\alpha!} \\
= \ & L^* d^{\frac{s_2}{2}}2^{-(l+1)s}\frac{d^{s_1}}{s_1!} =
L^* \frac{d^{s_1+\frac{s_2}{2}}}{s_1!} 2^{-s(l+1)},
\end{align*}
where we used the multinomial formula in the second to last equality.
\end{proof}

\begin{corollary}
\label{cor:mean_approximation_error}
In the setting of Theorem \ref{thm:regression_error_piecewise_pols}, we have
\begin{align}
\label{eq:approximation_error_mean_appendix}
\inf_{h \in \CF(\hat A, l, s_1, R)}\bbE\left(h(X)- f(X)\right)^2 \leq C_1 N^{-\frac{2s}{2s+d}} + C_2 \log^{1\wedge s}(N)\N{\MIMPTrue - \MIMP}^{2 \wedge 2s},
\end{align}
with $C_1$ depending on $L^*, d, s$ and $C_2$ depending on $L^*$ and linearly on $(D\N{X}_{\psi_2}^2)^{1 \wedge s}$.
\end{corollary}
\begin{proof}
Using the law of total expectation, and $\SN{h(X) - f(X)} = \SN{f(X)} \leq 1$ if $\N{X} > R$,
we obtain for any $h \in \CF(\hat A, l,s_1,R)$
\begin{equation}
\begin{aligned}
\label{eq:mim_app_aux_2_initial_error_decomposition}
\bbE\left(h(X)- f(X)\right)^2 &\leq \bbE \left[\left(h(X)- f(X)\right)^2\Big|\N{X} \leq R\right]\bbP(\N{X} \leq R) \\
&+ \bbE\left[\left(h(X)- f(X)\right)^2\Big| \N{X} > R\right]\bbP(\N{X} > R)\\
&\leq \bbE\left[\left(h(X)- f(X)\right)^2\Big|\N{X}\leq R \right] + \bbP(\N{X} > R).
\end{aligned}
\end{equation}
For the first term, we use the function $h$ in Proposition \ref{prop:close_function_h}
satisfying the guarantee \eqref{eq:mim_app_aux_approximation_error}.
Using $l=\lceil \log_2(N)/(2s+d)\rceil$, or $2^{-l} \geq N^{-1/(2s+d)}$, and $R^2 = D\N{X}_{\psi_2}^2\log(N)$ we get
\begin{align*}
& \bbE\left[\left(h(X)- f(X)\right)^2\Big|\N{X}\leq R \right] \\
\leq \ & \left(L^*\frac{d^{s_1+\frac{s_2}{2}}}{s_1!} 2^{-s(l+1)} + L^* R^{1\wedge s} \N{\MIMP-\MIMPTrue}^{1\wedge s}\right)^2 \\
\leq \ & \tilde C_1 2^{-2sl} +2 (L^*)^2 R^{2 \wedge 2s} \N{\MIMP-\MIMPTrue}^{2 \wedge 2s} \\
\leq \ & \tilde C_1  N^{-\frac{2s}{2s+d}} + 2(L^*)^2 \left(D\N{X}_{\psi_2}^2 \log(N)\right)^{1\wedge s} \N{\MIMP-\MIMPTrue}^{2 \wedge 2s}.
\end{align*}

\noindent
For the second term in \eqref{eq:mim_app_aux_2_initial_error_decomposition}, we note that $\N{X}$ is a sub-Gaussian
with $\N{\N{X}}_{\psi_2} \leq \sqrt{D}\N{X}_{\psi_2}$ by Lemma \ref{lem:norm_concentration}.
Therefore, using $R^2 = D\N{X}_{\psi_2}^2\log(N)$ we have by \cite[Proposition 2.5.2]{vershynin2018high}
\begin{equation*}
P(\N{X} > R) \leq \exp\left(-\frac{R^2}{D\N{X}_{\psi_2}^2}\right) \leq \exp(-\log(N)) = N^{-1}. \qedhere
\end{equation*}
\end{proof}
\paragraph*{Finalizing the argument}
\begin{proof}[Proof of Theorem \ref{thm:regression_error_piecewise_pols}]
Theorem \ref{lem:combined_bound} and Corollary \ref{cor:mean_approximation_error} imply
\begin{gather}
\begin{aligned}
\label{eq:mim_pp_proof_aux_1}
& \bbE\left(\hat f(X) - f(X)\right)^2 \\
\leq \ & C \max\{\sigma_{\zeta}^2, 1\}\frac{\log(N) + \dim(\CF)}{N} + C_1' N^{-\frac{2s}{2s+d}} + C_2'\log^{1\wedge s}(N) \N{\MIMP-\MIMPTrue}^{2 \wedge 2s},
\end{aligned}
\end{gather}
where $C_i' = C C_i$ with $C_i$ as in Corollary \ref{cor:mean_approximation_error}, and $C$ is a universal constant.
Furthermore, using Lemma \ref{lem:bound_number_of_cells}, $2^{l} \leq N^{1/(2s+d)} + 1$ and
$R^2 = D\N{X}_{\psi_2}^2\log(N)$, we bound the complexity of $\CF$ by
\begin{align*}
\dim(\CF(\hat A, l,s_1,R)) &\leq {{d+s_1}\choose s_1}\lceil(2^{l+1}R)^d\rceil \leq  2^d{{d+s_1}\choose s_1}\lceil(2^{l}R)^d\rceil \\
&\leq {{d+s_1}\choose s_1}2^d \left\lceil N^{\frac{d}{2s+d}} \left(D\N{X}_{\psi_2}^2\log(N)\right)^{\frac{d}{2}}\right\rceil \\
& \leq C_3'\log^{d/2}(N) N^{\frac{d}{2s+d}},
\end{align*}
with $C_3'$ depending on $d, s_1$ and linearly on $(D\N{X}_{\psi_2}^2)^{d/2}$.
Inserting this in \eqref{eq:mim_pp_proof_aux_1} and using $N^{\frac{d}{2s+d} - 1} = N^{-\frac{2s}{2s+d}}$,
the result follows for $C_2 = C_2'$ and $C_1 = \max\{C_1', C_3', C\max\{\sigma_{\zeta}^2,1\}\}$.
\end{proof}

\end{appendix}

\section*{Acknowledgements}
Timo Klock: this work has been carried out at Simula Research Laboratory (Oslo)
and has been supported by the Norwegian Research Council Grant No 251149/O70.
Stefano Vigogna: part of this work has been carried out at the Machine Learning Genoa (MaLGa) center, Universit\`a di Genova (IT),
and has been supported by the European Research Council (grant SLING 819789)
and the AFOSR projects FA9550-17-1-0390 and BAA-AFRL-AFOSR-2016-0007 (European Office of Aerospace Research and Development).


\bibliographystyle{imsart-number} 


\end{document}